\documentclass[]{article}
\usepackage{amssymb,amsmath,hyperref,verbatim,cite,graphicx}
\makeatletter \let\cl@chapter\relax \makeatother
\usepackage{cleveref}
\usepackage{multirow}
\usepackage{longtable}
\usepackage{pdflscape}
\usepackage{xcolor}
\usepackage{mathtools}
\usepackage[percent]{overpic}
\usepackage{amsthm}

\newtheorem{theorem}{Theorem}

\newtheorem{proposition}[theorem]{Proposition}
\newtheorem{corollary}[theorem]{Corollary}
\newtheorem{lemma}[theorem]{Lemma}
\theoremstyle{definition}
\newtheorem{definition}[theorem]{Definition}
\theoremstyle{remark}
\newtheorem*{remark}{Remark}
\theoremstyle{remark}
\newtheorem*{claim}{Claim}
\newenvironment{claimproof}[1]{\par\noindent\underline{Proof:}\space#1}{\hfill $\blacksquare$}
\usepackage{longtable}
\begin{document}



\allowdisplaybreaks

\title{Outer-Product-Free Sets for Polynomial Optimization and Oracle-Based Cuts}
\author{Daniel Bienstock\footnote{dano@columbia.edu, IEOR, Columbia University, New York, NY, USA} , Chen Chen\footnote{chen.8018@osu.edu, ISE, The Ohio State University, Columbus, OH, USA} , Gonzalo Mu\~{n}oz\footnote{gonzalo.munoz@polymtl.ca, Institute of Engineering Sciences, Universidad de O'Higgins, Rancagua, Chile; Conicyt BCH 72130388}\\}

\maketitle

\begin{abstract}
This paper introduces cutting planes that involve minimal structural assumptions, enabling the generation of strong polyhedral relaxations for a broad class of problems. We consider valid inequalities for the set $S\cap P$, where $S$ is a closed set, and $P$ is a polyhedron. Given an oracle that provides the distance from a point to $S$, we construct a pure cutting plane algorithm which is shown to converge if the initial relaxation is a polyhedron. These cuts are generated from convex forbidden zones, or $S$-free sets, derived from the oracle. We also consider the special case of polynomial optimization.  Accordingly we develop a theory of \emph{outer-product-free} sets, where $S$ is the set of real, symmetric matrices of the form $xx^T$. All maximal outer-product-free sets of full dimension are shown to be convex cones and we identify several families of such sets. These families are used to generate strengthened intersection cuts that can separate any infeasible extreme point of a linear programming relaxation efficiently. 
Computational experiments demonstrate the promise of our approach.

\end{abstract}

\section{Introduction}
Consider a generic mathematical program of the following form
\begin{alignat*}{2}
\min \ & c^Tx\\
\text{subject to } & x \in S \cap P.
\end{alignat*}

Here $P:=\{x\in \mathbb{R}^n|Ax\leq b\}$ is a polyhedral set, $c \in \mathbb{R}^n$ is a given cost vector, and $S \subset \mathbb{R}^n$ is a closed set. The natural linear programming (LP) relaxation $\min \{ c^Tx \, |\, x \in P \}$ provides a computationally tractable lower bound to the original problem's optimal objective value.  However, $P$ may be a poor outer-approximation or relaxation of the true feasible region $S \cap P$. This paper concerns the generation of stronger polyhedral relaxations via cutting-plane algorithm, i.e. the dynamic generation of cuts or valid (i.e. not removing points in $S \cap P$) linear inequalities to produce a sequence of (say, $k$) tighter relaxations: $P \supset P_1 \supset ... \supset P_k \supseteq S\cap P$. This cutting plane approach is crucial to branch-and-cut methods (e.g. \cite{padberg1991branch,belotti2013mixed,tawarmalani2005polyhedral,audet2000branch,misener2015dynamically}) for global optimization, and may be used to augment convex relaxations in general. 

There are many ways to generate cuts such as: disjunctions \cite{balas1998disjunctive}, lift-and-project \cite{lovasz1991cones}, algebraic arguments (e.g. \cite{gomory1972some,gomory1958outline,marchand2001aggregation,atamturk2010conic}), combinatorics (see \cite{wolsey2014integer}), and convex outer-approximation (e.g. \cite{kelley1960cutting}). We adopt the geometric perspective, in which cuts are derived from convex forbidden zones, or $S$-free sets. The convexity requirement on $S$-free sets is essential in standard intersection cuts \cite{balas1971intersection}, although nonconvex sets can be exploited in special cases (e.g. \cite{li2008cook}).   

The $S$-free approach was developed in the context of mixed-integer programming; we shall consider a different setting involving minimal structural assumptions on $S$.  Suppose there is an oracle that provides the distance from a point to $S$. For instance, such distance can be approximated to arbitrary accuracy in polynomial time in the case of integer variables (using rounding operations), polynomial optimization (using eigenvalues, see \Cref{sec:ballcut}), and cardinality constraints (see \Cref{subsec:card}). In \Cref{thm:convtheorem} we establish that, given the initial relaxation $P$ is a polytope, the oracle (or an arbitrarily close approximation thereof) enables a finite-time cutting plane algorithm that constructs a polyhedron arbitrarily close to $\mbox{conv}(S\cap P)$. Hence an explicit functional characterization of $S$ is not necessary to produce a strong relaxation. 

Additionally, we consider the case when $S$-free sets are used to derive an \emph{intersection cut}. Such a cut improves the relaxation in polynomial time by using a (weaker) basic relaxation of $P$. We develop a polynomial-time strengthening procedure for generic intersection cuts (see \Cref{sec:strengthen}) that exploits the recession cone of an $S$-free set. 

We also focus on the special case of polynomial optimization:
\begin{alignat*}{2}
\min\ &  p_0(x)  \nonumber\\
(\mathbf{PO})\ \ \text{s.t. }&p_i(x) \leq 0 &\quad \ i=1,...,m,
\end{alignat*}
where each $p_i$ is a polynomial function with respect to the decision vector $x \in \mathbb{R}^n$. Polynomial optimization generalizes important classes of problems such as quadratic programming, and has numerous applications in engineering.

 $\mathbf{PO}$ can be treated as a special case of mixed-integer nonlinear programming (MINLP). MINLP cuts are typically generated for a single nonlinear term or function (e.g. \cite{misener2012global,luedtke2012some,bao2009multiterm,rikun1997convex,tardella2008existence,locatelli2013convex,tawarmalani2002convex,tawarmalani2013explicit,serrano2019intersection}) over a simple subset of linear constraints such as box constraints. 
 In contrast, we develop general-purpose cuts that account for global nonconvexity imposed by $S$ (i.e. potentially addressing several nonlinear functions at once). To the best of our knowledge there are two papers (applicable to polynomial optimization) that are similar to our work in this regard: the disjunctive cuts of Saxena, Bonami, and Lee \cite{saxena2010convex,saxena2011convex}, which apply to bounded mixed-integer programming problems with nonconvex quadratic constraints (MIQCP);
 and the lift-and-project method by Ghaddar, Vera, and Anjos \cite{ghaddar2011dynamic}, where cuts for a given moment relaxation are generated using a higher-moment relaxation.  Polynomial-time separation for these procedures is not guaranteed in general.

We work with a representation of $\mathbf{PO}$ that uses a symmetric matrix of decision variables, and let $S$ be the set of symmetric matrices that can be represented as a real, symmetric outer product $xx^T$ ---accordingly we study \emph{outer-product-free} sets and the intersection cuts that can be derived from them. We first derive a simple oracle-based outer-product-free set in \Cref{sec:ellipsoids}. Subsequently, we identify several families of \emph{maximal} outer-product-free sets in \Cref{thm:psdhalf,thm:22thm}; such families are sufficient to characterize all such (full-dimensional) sets in the space of $2\times 2$ symmetric real matrices. With the aforementioned results we develop a cut generation procedure (see \Cref{sec:polycuts}) that separates an infeasible extreme point of a (lifted) polyhedral relaxation of $\mathbf{PO}$ in polynomial time without relying on variable bounds.

We demonstrate the practical effectiveness of our approach over a variety of instances using a straightforward pure cutting-plane setup. Comparisons are made with semidefinite programming relaxations, as well as the cuts of Saxena, Bonami, and Lee. The speed of our separation routines and the quality of the resulting linear programming relaxations strongly suggest the viability of our cut families within a full-fledged branch-and-cut solver.\\

The remainder of the paper is organized as follows.  \Cref{sec:intcutsold} describes $S$-free sets and develops oracle-based cuts. \Cref{sec:intcutsnew} describes the standard intersection cut, and our cut strengthening procedure. \Cref{sec:poly} studies outer-product-free sets. \Cref{sec:polycuts} describes cut generation using outer-product-free sets. \Cref{sec:exp} provides numerical examples and detailed computational experiments. \Cref{sec:conc} concludes.

\subsection{Notation}
Denote the interior of a set $\mbox{int}(\cdot)$, its boundary $\mbox{bd}(\cdot)$, its closure $\mbox{cl}(\cdot)$, and its recession cone $\mbox{rec}(\cdot)$.  The convex hull of a set is denoted $\mbox{conv}(\cdot)$, and its closure is $\mbox{clconv}(\cdot)$; likewise, the conic hull of a set is $\mbox{cone}(\cdot)$, and its closure $\mbox{clcone}(\cdot)$. The set of extreme points of a convex set is $\mbox{ext}(\cdot)$. For a point $x$ and nonempty set $S$ in $\mathbb{R}^n$, we define $d(x,S):=\inf_{s\in S}\{\|x-s\|_2\}$; note that for $S$ closed we can replace the infimum with minimum.  Denote the ball with center $x$ and radius $r$ to be $\mathcal{B}(x,r)$. The $i$th row of a matrix $A$ is $a_{i,*}$, and the $j$th column is $a_{*,j}$. For a square matrix $X$, $X_{[i,j]}$ denotes the $2\times 2$ principal submatrix induced by indices $i\neq j$. The $2\times 2$ submatrix of $X$ induced by rows $i_1 \neq i_2$ and columns $j_1\neq j_2$ is denoted $X_{[[i_1, i_2], [j_1, j_2]]}$. $\langle\cdot,\cdot\rangle$ denotes the matrix inner product and $\|\cdot \|_F$ the Frobenius norm. A positive semidefinite matrix may be referred to as a PSD matrix for short, and likewise NSD refers to negative semidefinite.

\section{$S$-free Sets and Oracle-Based Cuts}
\label{sec:intcutsold} 
\begin{definition}
A set $C\subset \mathbb{R}^n$ is $S$-\emph{free} if $\mbox{int}(C)\cap S =\emptyset$ and $C$ is convex. 
\end{definition}

For any $S$-free set $C$ we have $S \cap P \subseteq \mbox{clconv}(P\setminus \mbox{int}(C))$, and so any valid inequalities for $\mbox{clconv}(P\setminus \mbox{int}(C))$ are valid for $S \cap P$. See \Cref{fig:oracleballcut} for a diagram. Hillestad and Jacobsen \cite{hillestad1980reverse}, and later on Sen and Sherali \cite{sen1987nondifferentiable}, provide results regarding the polyhedrality of $\mbox{clconv}(P\setminus \mbox{int}(C))$. Averkov \cite{averkov2011finite} provides theoretical consideration on how one can derive cuts from $C$. In specific instances, $\mbox{conv}(P\setminus \mbox{int}(C))$ can be fully described; for example, Bienstock and Michalka \cite{bienstock2014cutting} provide a characterization of the convex hull when $S$ is given by the epigraph of a convex function excluding a polyhedral or ellipsoidal region (also see \cite{modaresi2015intersection,Belotti2,kilincc2014minimal}).

\begin{figure}
    \centering
    \begin{overpic}[scale=0.25,trim={5cm, 5cm, 5cm, 2.5cm},clip,page=3]{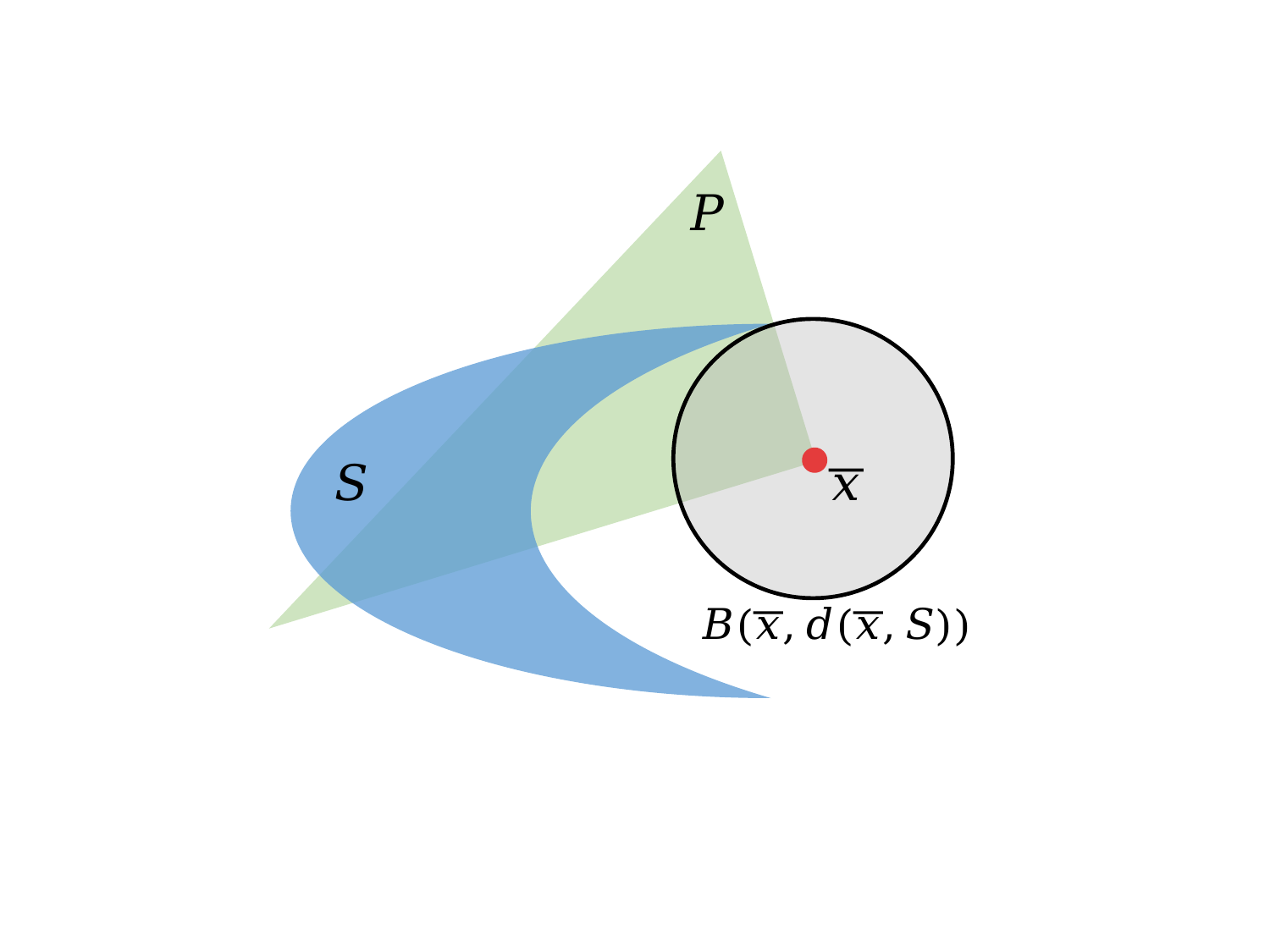}
    \put (55,60) {$P$}
    \put (10,25) {$S$}
    \put (75,25) {$\bar{x}$}
    \put (73,6) {$\mathcal{B}(\bar{x},d(\bar{x},S))$}
    \end{overpic}
    \hspace*{1cm}
    \begin{overpic}[scale=0.25,trim={5cm, 5cm, 5cm, 2.5cm},clip,page=4]{Simpledrawings.pdf}
     \put (15,60) {$P\setminus \mathcal{B}(\bar{x},d(\bar{x},S))$}
    \put (10,25) {$S$}
    \put (75,25) {$\bar{x}$}
    \end{overpic}
    \caption{On the left: a closed set $S$ (blue), a polyhedron $P$ (green), an extreme point $\bar{x}$ (red) and an $S$-free ball (grey). On the right: a hyperplane separating $\bar{x}$ from $P \setminus \mbox{int}(\mathcal{B}(\bar x,d(\bar x,S)))$}
    \label{fig:oracleballcut}
\end{figure}

In what follows we show a generic algorithm based on separation from $P\setminus \mbox{int}(C)$ that yields a convergent cutting plane algorithm using a simple $S$-free set $C$. Note, however, that separating over $P\setminus \mbox{int(C)}$ is NP-hard \cite{freund1985complexity}. 

\subsection{Oracle-Based Cuts}
\label{sec:ballcut} 
Suppose we have an oracle that provides for any given point $\bar x \notin S$ the nonzero Euclidean distance $d(\bar x,S)$ between $\bar x$ and the nearest point in $S$. 
\begin{remark}
The (closed) ball $\mathcal{B}(\bar x,d(\bar x,S))$ is $S$-free.  
\end{remark}

Suppose $P$ is a (bounded) polytope. We shall demonstrate that this $S$-free ball can be used to construct a pure cutting plane algorithm that will converge in the limit to the convex hull of $S \cap P$. Furthermore, an arbitrarily precise approximation of $\mathcal{B}(\bar x,d(\bar x,S))$ suffices to obtain an arbitrarily precise approximation of $\mbox{conv}(S \cap P)$.  This is not as strong as convergence in finite time, which can be established for simpler problems (e.g. \cite{gomory1963algorithm,porembski2002cone}); such a guarantee is not possible here since $\mbox{conv}(S\cap P)$ may be nonlinear. Finite convergence, however, is not strictly necessary for practical purposes in branch-and-cut (e.g. split cuts \cite{del2012convergence}).  Our result in this section contrasts with the standard intersection cut approach we describe below in \Cref{sec:intcutsnew}, in which good cuts are found by selecting large $S$-free sets and, instead of $P$, a (weaker) basic relaxation of $P$ is used. Here we show that it is possible to converge to $\mbox{conv}(S \cap P)$ by shifting the computational burden from selecting elaborate $S$-free sets to generating cuts over the entire polyhedron, i.e. separation over $\mbox{conv}(P \setminus \mbox{int}(\mathcal{B}(\bar x,d(\bar x,S))))$.

\subsection{Separation}

\label{sec:ballcuts}
We seek a cut $\alpha^T (x-\bar x) \geq \delta$ that separates $\bar x$ from $P\setminus \allowbreak \mbox{int}(\mathcal{B}(\bar x,d(\bar x,S)))$.  Such a cut can be determined via the following master cut generation problem,
\begin{subequations}
\begin{alignat}{2}
\delta^*(\bar x) \, \doteq \, \max_{\alpha,\delta} \ & \delta \nonumber \\
(\mathbf{MC})\ \ \text{s.t. }  & \alpha^T (x-\bar x) \geq \delta & \ \ \forall x \in \mbox{conv}(P\setminus \mbox{int}( \mathcal{B}(\bar x,d(\bar x,S)))),\label{eq:subball} \\
&\|\alpha\|_1 \leq 1. \label{eq:norm}
\end{alignat}
\end{subequations}

The cut normalization constraint~(\ref{eq:norm}) is replaceable, for instance, with the 2-norm. Norm selection has been subject to extensive testing and discussion in mixed-integer programming (e.g. \cite{fischetti2010note}), but we leave alternative formulations of $\mathbf{MC}$ out of this initial proposal.

\Cref{fig:strongball} demonstrates that separation involves more than one nontrivial facet in general, and indeed the problem is NP-hard \cite{freund1985complexity}. The increased computational expense, however, guarantees strong cuts that ensure favourable convergence properties; it also offsets the need to find an appropriate $S$-free set, which is also NP-hard (e.g. \cite{fischetti2007optimizing}).

Next, we argue that, given $\epsilon > 0$, problem $\mathbf{MC}$
can be solved to additive tolerance $\epsilon$ in finite time.  We rely on the following
observation whose proof is straightforward.

\begin{lemma} \label{lem:newpoints}
  Let $Q \subseteq \mathbb{R}^n$ be a polyhedron and consider a ball $\mathcal{B}(u, R)$.
  Then every extreme point $v$ of $\mbox{conv}(Q\setminus \mbox{int}(\mathcal{B}(u,R)))$ is either (a) an extreme point of $Q$, or (b) is contained in a $1$-dimensional face of $Q$, in which case $d(u,v) = R$.
\end{lemma}  

In view of this result, problem $\mathbf{MC}$ can be rewritten as a finite linear program, by simply enumerating all extreme points of $P\setminus \mbox{int}( \mathcal{B}(\bar x,d(\bar x,S)))$. This approach, however, does not yield a finite algorithm because extreme points of type (b) may have irrational
coordinates. This issue is resolved by replacing each type (b) extreme point $v$, with a rational point $\hat v$,
in the same 1-dimensional face of $P$ but slightly closer to $\bar x$.  In particular, given
$\epsilon > 0$ we can guarantee that $\|v - \hat v\|_1 \le \epsilon$. By applying this method to all
type (b) extreme points we
obtain a rational polyhedron $P_\epsilon(\bar x)$ containing $P\setminus \mbox{int}( \mathcal{B}(\bar x,d(\bar x,S)))$ such that
\begin{subequations}
\begin{alignat}{2}
\delta^*_\epsilon (\bar x) \, \doteq \, \max_{\alpha,\delta} \ & \delta \nonumber \\
(\mathbf{MC_\epsilon})\ \ \text{s.t. }  & \alpha^T (x-\bar x) \geq \delta & \quad  \forall x \in P_\epsilon(\bar x)\\
&\|\alpha\|_1 \leq 1
\end{alignat}
\end{subequations}
satisfies $\delta^*_\epsilon (\bar x) \le \delta^*(\bar x) \le \delta^*_\epsilon (\bar x) + \epsilon$.  
One can further argue that the extreme points of $P_\epsilon(\bar x)$, represented as rationals, require a number of digits that is polynomial in the size of the description of $P$, the number of digits in an $\epsilon$-approximation to $\log d(\bar x, S)$, and $\epsilon$. For related material, see \cite{Schrijver86}.

As an alternative to the above method, one can address problem $\mathbf{MC}$ using Benders decomposition\cite{benders1962partitioning}, by relying on a separation algorithm to handle constraint~(\ref{eq:subball}). Given a proposed candidate cut $(\hat \alpha, ~ \hat \delta)$, we wish to find a point $\hat x \in \mbox{conv}(P\setminus \allowbreak \mbox{int}(\mathcal{B}(\bar x,d(\bar x,S))))$ for which $\hat \alpha^T(\hat x - \bar x) <\hat \delta$, or else certify that the candidate cut is valid for constraint~(\ref{eq:subball}). This task may be formulated as the subproblem
\begin{subequations} \label{eq:sepproblem}
\begin{alignat}{2}
z_{\mbox{sc}}^* (\hat \alpha, \hat \delta) := & \ \max_x \ d(x,\bar x) \nonumber\\
(\mathbf{SC})\ \quad \text{s.t. } & \quad \hat \alpha^T (x-\bar x) \leq \hat \delta,\label{eq:obeycut}\\
&\quad x \in P, \label{eq:inP}
\end{alignat}
\end{subequations}

Clearly, the cut $(\hat \alpha, ~ \hat \delta)$ is valid for  $\mathbf{MC}$ iff $z_{\mbox{sc}}^*(\hat \alpha, \hat \delta)\leq d(\bar x,S)$. Problem \eqref{eq:sepproblem} can be solved within any desired tolerance in finite time
(by enumerating extreme points of $P$, or by using branch-and-bound \cite{locatelli2000finite}).

Thus, the \textit{separation problem} for $\mathbf{MC}$ can be (approximately) solved in finite time; using the machinery of the ellipsoid method (see \cite{Schrijver86,grotschel1981ellipsoid}) one thus obtains another finite method for solving $\mathbf{MC}$ to a given tolerance, in
finite time, and over the rationals.

\begin{figure}
  \centering
   \includegraphics[width=0.5\textwidth]{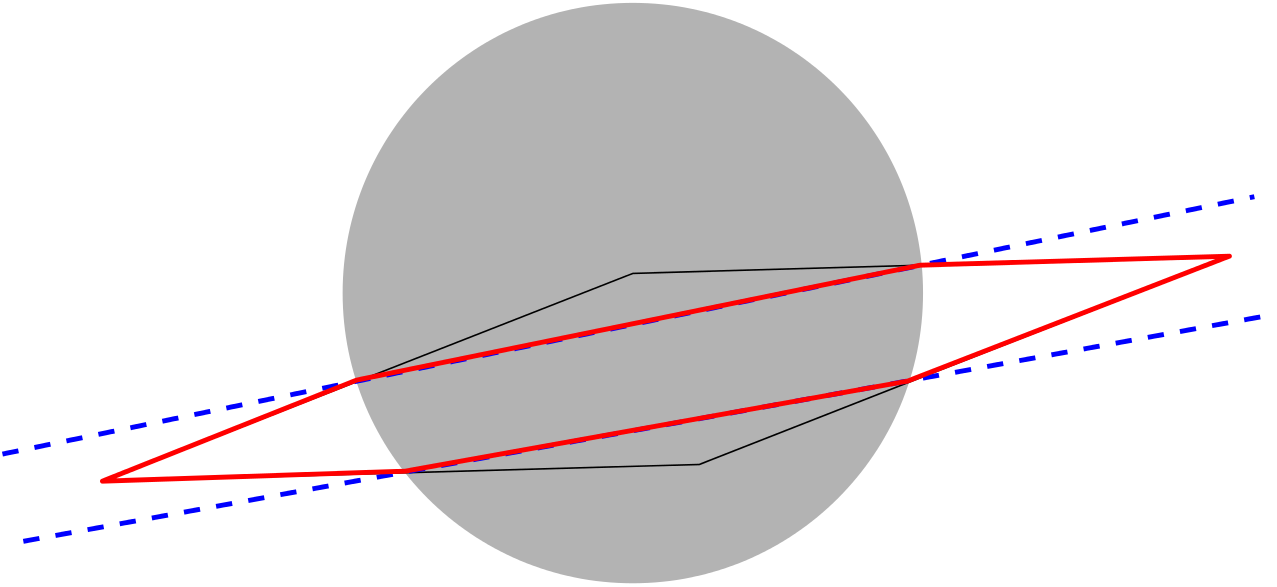}
      \caption{A parallelogram $P$ minus a ball $B$. The convex hull of $P\setminus B$ shown in thick red lines; its nontrivial facets are described by the cuts indicated by dotted blue lines.}
\label{fig:strongball}
\end{figure}

\subsection{Convergence of Cut Closures}
\label{sec:convclose}
We can (approximately) separate over $P\setminus \allowbreak \mbox{int}(\mathcal{B}(\bar x,d(\bar x,S)))$ by solving (approximately) $\mathbf{MC}$.  We now study the strength of such oracle-based cuts in two ways.  First, we consider the cut closure of such cuts, and show that a sequence of such closures converges to the best possible convex relaxation, $\mbox{conv}(P\cap S)$.  Second, we show that a cutting plane procedure, leveraging $\mathbf{MC}$, can converge to within arbitrary distance of $\mbox{conv}(P\cap S)$ in finite time.

\subsubsection{Cut Closures}
Throughout this subsection we assume that $P$ is bounded.  We follow closely the proof strategy of Averkov \cite[Theorem 3.6]{averkov2011finite}, which establishes convergence with respect to certain cuts given some (different) structural assumptions regarding $S$.  Our result applies to closed sets $S$ equipped with an oracle, which is a different domain of application than that of Averkov.  We also allow for a cutting plane procedure (in particular, procedure $\mathbf{MC}$) with fixed numerical precision, where separation is only guaranteed over a ball with radius exceeding some minimum threshold $\lambda \geq 0$. More precisely,
we assume that there is an oracle that, given $x \in \mathbb{R}^n$, returns an estimate $\tilde d(x,S)$ with $d(x,S) \le \tilde d(x,S) \le  d(x,S) + \lambda$.  This
yields an underestimate for $d(x,S)$:
$d(x,S) - \lambda \le \tilde d(x,S) - \lambda \le d(x,S)$. We will term the quantity
$\tilde d(x,S)$ a \emph{$\lambda$-overestimate} for $ d(x,S)$. Other notions of approximations for $d(x,S)$ are similarly handled.

The Hausdorff distance $d_H(X,Y) := \allowbreak \max\{\sup_{x\in X}d(x,Y),\allowbreak \sup_{y\in Y}d(y,X)\}$ between two sets $X,Y$ provides a natural way to describe convergence. An alternative definition of $d_H$ is available using the notion of $\epsilon$\emph{-fattening}. The $\epsilon$-fattening of a set $X$ is $X_\epsilon := \cup_{x\in X}\mathcal{B}(x,\epsilon)$, and so $d_H(X,Y) = \allowbreak \inf\{\epsilon\geq 0 | X \subseteq Y_\epsilon, Y \subseteq X_\epsilon\}$.  
Now let $P_0(\lambda) := P$, and define the rank $k$ closure (see  \cite{chvatal1973edmonds}) recursively as
\[P_{k+1}(\lambda) := \allowbreak \bigcap_{x\in\mbox{ext}(P_{k}(\lambda))} \allowbreak \mbox{conv}(P_{k}(\lambda) \setminus \allowbreak \mbox{int}(\mathcal{B}(x,\max\{d(x,S)-\lambda,0\})))\]

Furthermore define the compact convex set $P_{\infty}(\lambda) := \cap_{k=0}^{\infty} P_{k}(\lambda)$, which is the infinite rank cut closure. Two lemmas are used, with proofs that can be found in Schneider \cite{schneider2014convex}. The first lemma gives us Hausdorff convergence in the sequence of cut closures \cite[Lemma 1.8.2 \& p. 69 Note 4]{schneider2014convex}.  

\begin{lemma}
\label{lem:hausconv}
Let $(C_k)_{k\in\mathbb{N}}$ be a sequence of nonempty compact sets in $\mathbb{R}^n$, and denote $C_\infty := \cap_{i=0}^\infty C_k$.  If $C_k \supseteq C_{k+1} \forall k$ then it holds that $C_\infty = \lim_{k\to\infty}C_k$ and $\lim_{k\to\infty} d_H(C_k,C_\infty)=0$.
\end{lemma}

\begin{corollary}
\label{cor:finite}
Let $(C_k)_{k\in\mathbb{N}}$ and $C_\infty$ be as in Lemma \ref{lem:hausconv}. For each 
$\epsilon > 0$ there exists $k_{\epsilon}$ such that for all $k \ge k_{\epsilon}$,
$d_H(C_k, C_\infty) \le \epsilon$.
\end{corollary}

The next lemma \cite[Lemma 1.4.6]{schneider2014convex} ensures the existence of a ball cut that can separate an extreme point of a convex relaxation.
\begin{lemma}
\label{lem:caplem}
Let $C \subset \mathbb{R}^n$ be a closed, convex set and let $x \in C$.  Then $x$ is an extreme point of $C$ iff for every open neighbourhood $U$ around $x$ there exists a hyperplane $H$ defining the boundary of two (separate) halfspaces $H^-,H^+$ such that $x\in \textnormal{int}(H^-), C\setminus U \subseteq \textnormal{int} (H^+)$.
\end{lemma}

\begin{theorem}
\label{thm:convtheorem}
$P_{\infty}(\lambda) \subseteq \textnormal{conv}(P\cap S_\lambda)$, where $S_\lambda$ is the $\lambda$-fattening of $S$.
\end{theorem}
\begin{proof}
By construction $P_{\infty}(\lambda) \subseteq P_0(\lambda)=P$, so it is sufficient to show that $\mbox{ext}(P_{\infty}(\lambda)) \in S_\lambda$.  We shall do so by way of contradiction. Suppose there exists $\bar x \in \mbox{ext}(P_{\infty}(\lambda))$ such that $\bar x \notin S_\lambda$; observe that $\bar x \notin S_\lambda$ implies $d(\bar x,S)-\lambda>0$.  Then let $U$ be an open ball of radius $r := \allowbreak (d(\bar x,S)-\lambda)/3$ centered at $\bar x$. By \Cref{lem:caplem} there exist two opposite-facing halfspaces $H^+,H^-$ such that $\bar x \in \textnormal{int}(H^-)$ and $P_{\infty}(\lambda) \setminus U \subseteq \textnormal{int}(H^+)$.  Since $U$ is open and $P_{\infty}(\lambda)\cap H^-$ is in the interior of $U$ (as otherwise $P_{\infty}(\lambda) \setminus U \subseteq \textnormal{int}(H^+)$ is violated), there exists sufficiently small $\epsilon>0$ such that $(P_{\infty}(\lambda))_\epsilon \cap H^-$ is also contained in $U$. Now  \Cref{lem:hausconv} gives us some rank $k_0$ for which we have the sandwich $(P_{\infty}(\lambda))_\epsilon \supseteq P_{k_0}(\lambda) \supseteq P_{\infty}(\lambda)$.  Furthermore, since $H^-$ separates an extreme point of $P_{\infty}(\lambda)$, it also separates an extreme point of the superset $P_{k_0}(\lambda)$.  Thus there exists some extreme point $x_{k_0}\in\mbox{ext}(P_{k_0}(\lambda))$ that lies in $P_{k_0}(\lambda)\cap H^-\subset U$, and so $d(x_{k_0},\bar x)<r$.  Thus we have, starting from the triangle inequality 
\begin{alignat*}{2}
&d(x_{k_0}, S) + d(x_{k_0},\bar x) &\ \geq \ & d(\bar x, S)\\
\implies &d(x_{k_0},S)-\lambda &\ \geq \ & d(\bar x,S)-d(x_{k_0},\bar x)-\lambda,\\
&& \ >\  & d(\bar x,S)-\lambda-r,\\
& &\ = \ &2r.
\end{alignat*}

Since $U$ has diameter $2r$, then $U \subset \mbox{int}(\mathcal{B}(x_{k_0},d(x_{k_0},S)-\lambda))$.  As $H^+$ is valid for $P_{k_0}(\lambda)\setminus U$, then $H^+$ is also valid for the nested set $P_{k_0}(\lambda)\setminus \mbox{int}(\mathcal{B}(x_{k_0},d(x_{k_0},S)-\lambda))$.  It follows that $P_{k_0+1}(\lambda) \subseteq H^+$. However, $\bar x \in H^-$, which implies $\bar x \notin P_{k_0+1}(\lambda) \supseteq P_{\infty}(\lambda)$, giving us a contradiction. \qed
\end{proof}  

\subsubsection{Cutting-plane Procedure}
We now provide a formal cutting-plane procedure:

\begin{center}
  \fbox{
    \begin{minipage}{0.9\linewidth}
 \hspace*{.9in} {\bf Procedure CUT} \\
         {\bf Initialization:} Set $\hat P_0 = P$ and $k = 0$.\\
         {\bf Repeat:} \\
         {\bf 1.} Find an extreme point $\bar x$ of $\hat P_k$ such that $$\delta^*(\bar x) \ = \ \mbox{argmax}\{\delta^*(u) \, : u \in \text{ext}(\hat P_k)\}$$
         {\bf 2.} If $\delta^*(\bar x) \le \epsilon$, {\bf STOP.} Algorithm exits.\\
         {\bf 3.} Let $(\hat \alpha, \hat \delta)$ be optimal in the computation of $\delta^*(\bar x)$.\\
         Add the cut $\hat \alpha^T(x - \bar x) \ge \hat \delta$ to $\hat P_k$ to obtain $\hat P_{k+1}$.\\
         Set $k \leftarrow k+1$, and go to 1.
    \end{minipage}
  }
\end{center}
\begin{theorem}
  \label{thm:finiteconvcut} Procedure CUT terminates after a finite number of iterations $K$, and at termination $d_H(\hat P_K, \textnormal{conv}(P\cap S)) \le \sqrt{n} \epsilon$.  Consequently, $\hat P_K \subseteq \mbox{conv}(P\cap S_{\sqrt{n}\epsilon})$.
  \end{theorem}
\begin{proof} We first note that at each iteration of CUT where Step 3 is executed, the added
  inequality cuts off all points $x$ with $d(x,\bar x) \le \sqrt{n} \epsilon$ by Cauchy-Schwarz inequality.

  Now suppose, aiming
  for a contradiction, that CUT does not terminate finitely. Hence there is an infinite sequence
  $x_{k(i)}$, $i = 1, 2, \ldots$ where $x_{k(i)}$ is an extreme point of $\hat P_{k(i)}$. Let $\hat P_{\infty}$ be the limit of the $\hat P_k$.  Thus for $i$ large enough, $d_H(x_{k(i)}, \hat P_{\infty}) < \sqrt{n} \epsilon$.
  This contradicts the point made above.  Hence convergence is finite and another application
  of the same idea yields the second claim.
  
  Since CUT is initialized with $P$, we have $\hat P_K \subseteq P$, and so $d_H(\hat P_K, \textnormal{conv}(P\cap S)) \le \sqrt{n} \epsilon$ implies the final claim.
  \qed \end{proof}

\section{Intersection Cuts}
\label{sec:intcutsnew}
As mentioned in \Cref{sec:ballcuts}, separating over $P\setminus \mbox{int}(C)$ is NP-hard \cite{freund1985complexity}.  A standard workaround is to find a simplicial cone $P'$ containing $P$ and separate over $P'\setminus \mbox{int}(C)$ instead. Provided the apex of $P'$ is contained in the interior of $C$, Balas' intersection cut \cite{balas1971intersection} can be applied. In this section, we describe such cuts and provide a refined strengthening procedure for them.

The first use of simplicial cones to generate cuts can be attributed to Tuy \cite{tuy1964} for minimization of a concave function over a polyhedron; such cuts are named Tuy cuts, concavity cuts, or convexity cuts. The distinction is that Tuy cuts are objective cuts whereas Balas' intersection cuts are feasibility cuts. 

Larger $S$-free sets can generate deeper cuts \cite{conforti2014cut}, which leads to the notion of inclusion-wise maximality.

\begin{definition}
An $S$-free set $C$ is \emph{maximal} if $C'\not\supset C$ for all $S$-free $C'$.
\end{definition} 

Under certain conditions (see \cite{conforti2014cut,cornuejols2013sufficiency,kilincc2014minimal,basu2010minimal}), maximal $S$-free sets are sufficient to generate all nontrivial cuts for a problem.  When $S=\mathbb{Z}^n$, $C$ is called a lattice-free set.  Maximal lattice-free sets are well-studied in integer programming theory \cite{basu2010maximal,andersen2010analysis,borozan2009minimal,andersen2007inequalities,lovasz1989geometry,dey2008lifting,gomory1972some,kilincc2014minimal}, and the notion of $S$-free sets was introduced as a generalization \cite{dey2010constrained}.  

Many cuts in mixed-integer linear programming can be interpreted as intersection cuts, as intersection cuts produce all nontrivial facets of the corner polyhedron \cite{conforti2010equivalence}. 
Several papers \cite{dadush2011split,andersen2013intersection,modaresi2015split,modaresi2015intersection} have worked to extend the intersection cut via split cuts to mixed-integer conic optimization. Intersection cuts have also been considered for bilevel optimization \cite{fischetti2017new}, and factorable MINLPs \cite{serrano2019intersection}. Towle and Luedtke \cite{towle2019intersection} consider cuts for reverse convex sets that exploit $S$-free sets that do not contain the apex of $P'$.

\subsection{Classical derivation}
\label{sec:classicintcut}
Let $P' \supseteq P$ be a simplicial conic relaxation of $P$: a displaced polyhedral cone with apex $\bar x$ that is defined by the intersection of $n$ linearly independent halfspaces. Any $n$ linearly independent constraints describing $P$ can be used to define a simplicial conic relaxation, i.e. a (possibly infeasible) basis of $P$.  

A simplicial cone may be written as follows:
\begin{equation}
P' = \{ \bar x+\sum_{j=1}^n\lambda_j r^{j} \colon \lambda \geq 0\}. \label{eq:charsimcone}
\end{equation}

Each extreme ray of $P'$ is of the form $\{\bar x + \lambda_jr^{j}|\lambda_j \geq 0\}$. 
Alternatively, the simplicial conic relaxation can be given in inequality form 
\begin{equation}
    \label{eq:simconeineq}
    P'=\{x|\bar Ax\leq \bar b\},
\end{equation} 
where $\bar A$ is an invertible $n \times n$ submatrix of $A$, and $\bar b$ are the corresponding entries of $b$. Note that any basis of $P$ would be suitable to derive $P'$. In this case the apex is $\bar x = \bar A^{-1}\bar b$, and the rays $r^{j}$ in \eqref{eq:charsimcone} can be obtained directly from $\bar A$: for each $j$, one can identify $-r^{j}$ as the $j$th column of the inverse of $\bar A$.

We shall assume $\bar x \notin S$, so that $\bar x$ is to be separated from $S$ via separation from $P'\setminus \mbox{int}(C)$, with  $C$ an $S$-free set with $\bar x$ in its interior. Since $\bar x \in \mbox{int}(C)$, there must exist $\hat \lambda>0$ such that $\bar x + \hat \lambda_jr^{j} \in \mbox{int}(C)\ \forall j$.  Also, each extreme ray is either entirely contained in $C$, i.e. $\bar x + \lambda_jr^{j} \in \mbox{int}(C) \, \forall \lambda_j\geq 0$, or else there is a finite intersection point with the boundary, $\exists \lambda^*_j : \bar x + \lambda^*_jr^{j} \in \mbox{bd}(C)$. We refer to $\lambda^*_j$ as the \emph{step length} in the latter case, and for convenience, we define the step length $\lambda_j^* = \infty$ in the former case. 

Let $M$ be the index set of finite step lengths, i.e. $\lambda_m <\infty$ $\forall m \in M$; likewise the complement $\bar M$ is the set of indices of infinite step lengths. The \emph{intersection cut} is the halfspace whose boundary contains each finite intersection point (indexed by $M$) that is parallel to all extreme rays contained in $C$ (indexed by $\bar M$), and that separates $\bar x$. Given $\lambda_j^*\in (0,\infty]\, \forall j = 1, \ldots, n$, Balas \cite[Theorem 2]{balas1971intersection} provides a closed-form expression for the intersection cut $\pi^T x \leq \pi_0$:
\begin{equation}
\pi_0 = \sum_{ i=1}^n (1/\lambda_i^*) \bar b_i -1,\ \pi_j = \sum_{i=1}^n (1/\lambda_i^*) \bar a_{i,j},  \label{eq:closedformcut}
\end{equation}
where $\bar a_{i,j}$ are the entries of $\bar A$ in \eqref{eq:simconeineq} and $1/\infty \coloneqq 0$ \cite[p. 34]{balas1971intersection}. Rearranging terms, one can see the intersection cut is formed using weighted rows of $\bar Ax-\bar b$:
 \begin{equation}
 \sum_{i=1}^n (\bar a_{i,*}x-\bar b_i)/\lambda_i^* \leq - 1.    \label{eq:closedformcut_2}
 \end{equation}
Hence the apex $\bar x=\bar A^{-1}\bar b$ violates the cut by a normalized value of $1$.

Let $V:=\{x|\pi^Tx\leq \pi_0\}$ be the halfspace defined by the intersection cut.  For completeness we include a proof of validity of $V$ (a fact established in the original paper by Balas \cite[Theorem 1]{balas1971intersection}), and furthermore establish a condition in which the cut gives the convex hull of $P'\setminus \textnormal{int}(C)$.

\begin{proposition}
\label{prop:validint}
$V \supseteq P'\setminus \textnormal{int}(C)$.  Furthermore, if all step lengths are finite, i.e. $|\bar M| = 0$, then $V\cap P'=\textnormal{conv}(P'\setminus \textnormal{int}(C))$.
\end{proposition}
\begin{proof}
Let $\bar V$ be the complement of $V$, i.e. an open halfspace. It suffices to establish that $Q := \bar V \cap P' \subseteq \mbox{int}(C)$. By construction of the cut, 

\[Q = \{\bar x+\sum_{i\in M}\lambda_i r^{i}+\sum_{j\in \bar M}\lambda_j r^{j}\, | \, \lambda \geq 0, \lambda_i < \lambda_i^*\ \forall i \in M\}. \]
Let $\hat x$ be a point in $Q$.  Denote $v:= \bar x+\sum_{i\in M}\hat \lambda_i r^{i}$ and $w:=\sum_{j\in \bar M}\hat \lambda_j r^{j}$ so that $\hat x = v+w$. Now $v$ is in the polytope $P'':=\{\bar x + \sum_{i\in M}\lambda_i r^{i}\, | \, 0\leq \lambda\leq \lambda^*\}$.  The extreme points of $P''$ are $\bar x$ and some intersection points forming the intersection cut. The vector $v$ cannot be described solely as the convex combination of intersection points, as this would imply $v \in V$. Hence either $v=\bar x$ or $v$ is the strict convex combination of $\bar x \in \mbox{int}(C)$ and some of the intersection points, thus $v \in \mbox{int}(C)$.  

Observe that $w$ is in the recession cone of $C$ since by construction each extreme ray indexed by $\bar M$ is contained in $C$. Hence $\hat x = v + w$ is also in the interior of $C$, or equivalently $Q \subseteq \mbox{int}(C)$, which in turn implies $V \supseteq P'\setminus \textnormal{int}(C)$.

Now suppose all step lengths are finite. In this case we may write
\[V\cap P' =\{\bar x + \sum_{i=1}^n \lambda_ir^{i}\, | \, \lambda \geq \lambda^*\} =\{\sum_{i=1}^n\lambda_i(\bar x + \sum_{j=1}^n \lambda_jr^{i})/\sum_{j=1}^n\lambda_j\, | \, \lambda\geq \lambda^*\}.\]
Hence every point in $V\cap P'$ is the convex combination of points $p_i := \bar x + \sum_{j=1}^n\lambda_jr^{i},\ i=1,...,n$. Since $\sum^n_{j=1}\lambda_j \geq \lambda_i^* \, \forall i$ and the ray $r^{i}$ emanates from an interior point of $C$ passing through the boundary at step length $\lambda_i^*$, then $p_i\in P'\setminus \mbox{int}(C)\ \forall i$.  Hence $V\cap P' \subseteq \mbox{conv}(P'\setminus \mbox{int}(C))$. \qed
\end{proof}

\Cref{fig:oracleballcut} also serves as an illustration of an intersection cut when the $S$-free set $C$ is a ball. In this example, using a simplicial conic relaxation $P'$ obtained from the basis defining $\bar{x}$ yields an intersection cut which coincides with the oracle-based cut depicted in the figure.

\subsection{Strengthening the Intersection Cut}
\label{sec:strengthen}
The intersection cut is not in general sufficient to capture the convex hull of $P'\setminus \mbox{int}(C)$.  When $|\bar M|\geq 1$, negative step lengths $y<0$ can be used to strengthen intersection cuts. Glover \cite{glover1974polyhedral} proposed a method to derive such negative step lengths for polyhedral $C$, later extended by Sen and Sherali \cite{sen1987nondifferentiable} to general reverse convex programs.  Negative step lengths have also been considered in the context of minimal valid inequalities \cite{dey2010constrained,basu2011convex,basu2010maximal}. We provide here a formula for general-purpose strengthening that provides the provably best rotation of the cut with the finite intersection points fixed.  Furthermore, the coefficients can be calculated with a polynomial number of single-variable optimization problems over the recession cone of $C$. 


\subsubsection{Motivating Example}
A simple example of intersection cut strengthening is shown in \Cref{fig:exstr2}.  Here $S\subset \mathbb{R}^2$ is given by the halfspace in $x_1+x_2 \geq 1$.  We let $C$ be the (unique) maximal $S$-free set, $x_1+x_2 \leq 1$. Now define $P'$ with inequalities $x_1 \geq 0$ and $x_2 \leq 0$.  The apex of $P'$ is at the origin, and the extreme ray directions are $r^{1} = (1,0), r^{2} = (0,-1)$. The intersection points with $C$ are at $(1,0)$ along $r^{1}$ and infinity along $r^{2}$. The intersection cut is given by $x_1 \geq 1$.  The best cut possible from $C$, however is $x_1+x_2 \geq 1$, i.e. the set $S$ itself.  To obtain this cut from $P'$ one can use the cut passing through the (original) finite intersection point $(1,0)$ together with $(0,1)$, which can be obtained by taking a step of length $-1$ from the origin along $r^{2}$.  This strengthened cut can also be viewed as a standard intersection cut generated from a tightening of $P'$ where $r^2$ is rotated to $r^2_*$.

\begin{figure}
  \centering
   \includegraphics[width=0.4\textwidth]{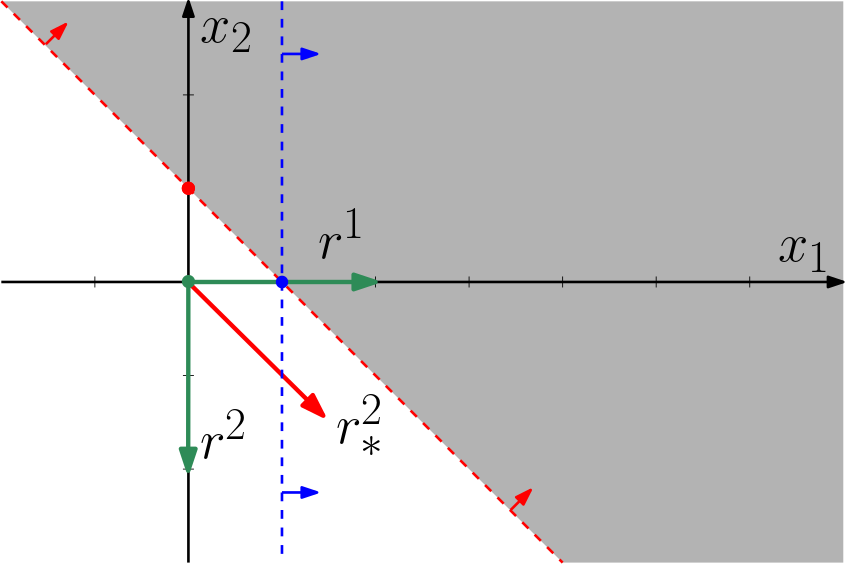}
      \caption{Example of cut strengthening.  In green, simplicial cone $P'$ with rays $r^1,r^2$; origin marked with green dot.  In grey, set $S$, and the $S$-free set $C$ is the complement of $S$.  Blue dashes indicate the standard intersection cut;  a blue dot marks the intersection point with $\mbox{bd}(C)$.  The strengthened intersection cut is shown with red dashes; a red dot marks the intersection point obtained with a negative step length. The rotated extreme ray $r^2 \to r^2_*$ is shown in red.}
\label{fig:exstr2}
\end{figure}

\subsubsection{Strengthening Procedure} \label{sec:strengthening}
The index set of finite step lengths $M$ is assumed to be nonempty, as otherwise the original problem is proven to be infeasible, i.e. $S\cap P = \emptyset$. Our strengthened intersection cut is of the form
\begin{equation}
\sum_{m\in M} (\bar a_{m,*}x-\bar b_m)/\lambda_m^* +\sum_{j \in \bar M} (\bar a_{j,*}x-\bar b_j)/y_j  \leq - 1     \label{eq:negcut}
\end{equation}
where $y_j < 0$ is a negative step length. Each negative step length is a parameter that rotates the intersection cut about the axis defined by the other $(n-1)$ fixed intersection points (including intersections at infinity). The strongest such cut is obtained with maximal valid rotation, i.e. (algebraically) maximal values of $y$. We prove in \Cref{thm:validstr} that such values are given by:
\begin{equation}
\label{eq:strength}
y^*_j :=  \max \{\gamma\, |\, \lambda^*_mr^{m} - \gamma r^{j} \in \mbox{rec}(C)\quad \forall m \in M\}.
\end{equation}
If no such $\gamma$ exists we take $y^*_j=-\infty$ with the understanding that $1/y^*_j=1/\lambda_j^*=0$. Note that $y^*<0$ by the following argument.  Suppose $\gamma \geq 0$ and furthermore for some $m \in M$ we have $w_m := \lambda^*_mr^{m} - \gamma r^{j} \in \mbox{rec}(C)$, then $r^{m}$ is a conic combination of $w_m$ and $r^{j}$.  But this is impossible since $r^{m} \notin \mbox{rec}(C)$.

Now let $K \subseteq \bar M$ be the set of strictly negative finite step lengths given by \eqref{eq:strength}, i.e. $-\infty <y^*_k<0 \ \forall k \in K$. Furthermore, let $V_y$ be the halfspace defined by the strengthened cut \eqref{eq:negcut}, parameterized by negative step lengths $y$.

\begin{theorem}
\label{thm:validstr}
$V_y \cap P' \supseteq P'\setminus \mbox{int}(C)$ if and only if $y \leq y^*$.
\end{theorem}
\begin{proof}
\Cref{prop:validint} establishes validity of the original intersection cut \eqref{eq:closedformcut}.  Let $W := (V\cap P')\setminus V_y$ with $V$ the halfspace defined by the original intersection cut, i.e, $W$ is the space removed by strengthening the intersection cut. 

First, note that we always have $V_y \cap P' \subseteq V\cap P'$. We claim proving the theorem amounts to determining whether $W \subseteq \mbox{int}(C)$; if so
\[ V_y \cap P' = V\cap P' \setminus W \supseteq V\cap P'\setminus \mbox{int}(C)=P'\setminus \mbox{int}(C),\] 
where the last equality is given by \Cref{prop:validint}. If not, then $V_y$ removes points from $P'\setminus\mbox{int}(C)$.

Define the index set $K_y := \{k\, |\, y_k > -\infty\}$. Throughout we assume the nontrivial case that $|K_y|\geq 1$, as otherwise the strengthened intersection cut is equivalent to the original cut. The set $\mbox{cl}(W)$ is described by:
\begin{subequations}
\begin{alignat}{2}
&\bar Ax \leq \bar b \label{eq:sys1}\\
&\sum_{m\in M} (\bar a_{m,*}x-\bar b_m)/\lambda_m^*  \leq - 1 \label{eq:sys2}\\
&\sum_{m \in M} (\bar a_{m,*}x-\bar b_m)/\lambda_m^* +\sum_{k \in K_y} (\bar a_{k,*}x-\bar b_k)/y_k  \geq - 1.\label{eq:sys3}
\end{alignat}
\end{subequations}
Note that $W$ is described by the above system with (\ref{eq:sys3}) strictly satisfied. We shall now characterize the extreme points and extreme rays of $\mbox{cl}(W)$.  \\

\noindent\emph{Extreme points of $\mbox{cl}(W)$:}\\
The extreme points are determined by $n$ linearly independent binding constraints among \eqref{eq:sys1}-\eqref{eq:sys3} (with all remaining constraints satisfied). If $n$ inequalities are binding among \eqref{eq:sys1} we recover $\bar x$, i.e. the apex of $P'$, which violates (\ref{eq:sys2}) (it is cut off by $V$).  Therefore at least one of (\ref{eq:sys2}) or (\ref{eq:sys3}) is binding at each extreme point. Additionally, since $1/\lambda_m^* >0 \ \forall m \in M$, in order for (\ref{eq:sys2}) to be satisfied there must exist some $\hat m\in M$ such that $\bar a_{\hat m,*}x-\bar b_{\hat m}<0$. We analyze two cases: 

\noindent\underline{Case 1: (\ref{eq:sys2}) is binding.} If, additionally, (\ref{eq:sys3}) is not binding, then $\sum_{k \in K_y}(\bar a_{k,*}x-\bar b_k)/y_k  > 0$. Then, since $y<0$, there must exist some $\hat k \in K_y$ such that $\bar a_{\hat k,*}x-\bar b_{\hat k}<0$.  Now, $n-1$ constraints among $\bar Ax\leq \bar b$ must be binding, which implies $\hat m =\hat k$; however, this is impossible since $M$ and $K_y \subseteq \bar M$ are disjoint. Thus, (\ref{eq:sys3}) is also binding, which implies
\[\sum_{k \in K_y} (\bar a_{k,*}x-\bar b_k)/y_k  = 0.\]
Furthermore, given $\bar Ax\leq \bar b$ and $y<0$, we have $\bar a_{k,*}x-\bar b_k = 0 \ \forall k \in K_y$.  However, since (\ref{eq:sys3}) is a linear combination of (\ref{eq:sys2}) and the rows of (\ref{eq:sys1}) indexed by $K_y$ (all of which are binding), we conclude that we must have $n-1$ binding constraints among (\ref{eq:sys1}) at an extreme point of $\mbox{cl}(W)$. Indeed, for each $\hat m\in M$, there is a corresponding extreme point determined by $\bar a_{i,*}x=\bar b_i, i \in M, i \neq \hat m$ (from (\ref{eq:sys1})) and $(\bar a_{\hat m,*}x-\bar b_{\hat m})/\lambda_{\hat m}^*=-1$ (from (\ref{eq:sys2})). These coincide with the extreme points of $V \cap P'$, i.e. finite intersection points of $P'$ with $C$. 

\noindent\underline{Case 2: (\ref{eq:sys2}) is not binding.} In this case (\ref{eq:sys3}) is binding,  i.e.
\begin{alignat*}{2}
&\sum_{m \in M} (\bar a_{m,*}x-\bar b_m)/\lambda_m^* +\sum_{k \in K_y} (\bar a_{k,*}x-\bar b_k)/y_k  = - 1,\\
\implies & \sum_{k \in K_y} (\bar a_{k,*}x-\bar b_k)/y_k > 0.
\end{alignat*}
Then there must exist some $\hat k \in K_y$ such that $\bar a_{\hat k,*}x-\bar b_{\hat k}<0$.  However, there also exists  $\hat m\in M$ such that $\bar a_{\hat m,*}x-\bar b_{\hat m}<0$, leaving only $n-2$ constraints that can be binding among (\ref{eq:sys1}).  So there are no extreme points in this case.\\

We conclude that the extreme points of $\mbox{cl}(W)$ are the extreme points of $V \cap P'$, which by construction of $V$ are in $\mbox{bd}(C)$.  Furthermore, (\ref{eq:sys3}) is binding at all extreme points, so none of the extreme points are in $W$.\\ 

\noindent\emph{Extreme rays of $\mbox{cl}(W)$:}\\
The extreme rays are determined by $n-1$ linearly independent binding constraints among the following system (with all remaining constraints satisfied), which describes $\mbox{rec}(\mbox{cl}(W))$:
\begin{subequations}
\begin{alignat}{2}
&\bar Ad \leq 0 \label{eq:sys4}\\
&\sum_{m \in M} \bar a_{m,*}d/\lambda_m^*  \leq 0 \label{eq:sys5}\\
&\sum_{m \in M} \bar a_{m,*}d/\lambda_m^* +\sum_{k \in K_y} \bar a_{k,*}d/y_k  \geq 0.\label{eq:sys6}
\end{alignat}
\end{subequations}

Let $\hat d$ be an extreme ray of $W$. We shall characterize when $\hat d \in \mbox{rec}(C)$. If $\hat d \in \mbox{cone}(\{r^{j}|j \in \bar M\})$, by construction of $\bar M$ we have $\hat d \in \mbox{rec}(C)$. Now suppose instead that $\hat d$ is outside of said cone. This implies that there exists some $\hat m \in M$ such that $\bar a_{\hat m,*}\hat d<0$, and so (\ref{eq:sys5}) is not binding.  Then for (\ref{eq:sys6}) to be satisfied, we require some $\hat k \in K_y$ such that $\bar a_{\hat k,*}\hat d<0$. This leaves 2 constraints (indexed by $\hat m, \hat k$) nonbinding among (\ref{eq:sys4}) and so the remaining $n-2$ constraints, together with (\ref{eq:sys6}), are binding. 
From (\ref{eq:sys4}) we have that $\hat d$ is in the cone generated by $r^{\hat m}$ and $r^{\hat k}$. Recalling from \Cref{sec:classicintcut} that such rays are given by the columns of $-\bar A^{-1}$, we may write 
$\hat d=-\beta_{\hat m}\bar A^{-1}e_{\hat m} -\beta_{\hat k} \bar A^{-1}e_{\hat k}$, for some $\beta_{\hat m},\beta_{\hat k} \geq 0,$ with $e_i$ the $i$-th canonical vector. Since \eqref{eq:sys6} is binding we have:
\[\bar a_{\hat m,*}\hat d/\lambda_{\hat m}^* + \bar a_{\hat k,*}\hat d/y_{\hat k} =0,\]
\[\implies \bar a_{\hat m,*}(-\beta_{\hat m}\bar A^{-1}e_{\hat m} -\beta_{\hat k}\bar A^{-1}e_{\hat k})/\lambda_{\hat m}^* + \bar a_{\hat k,*}(-\beta_{\hat m}\bar A^{-1}e_{\hat m} -\beta_{\hat k}\bar A^{-1}e_{\hat k})/y_{\hat k} =0.\]
Now from $\bar A\bar A^{-1}=I$ we have $\bar a_{\hat m,*}\bar A^{-1} = e_{\hat m}^T, \bar a_{\hat k,*}\bar A^{-1} = e_{\hat k}^T$, and so
\[-\beta_{\hat m}  =\beta_{\hat k}\lambda_{\hat m}^*/y_{\hat k} \quad \implies \hat d=\beta_{\hat k}(r^{\hat k}-\frac{\lambda_{\hat m}^*}{y_{\hat k}} r^{\hat m}).\]


Suppose $y\leq y^*$, and so $K_y \subseteq K$.  Then by construction of $y^*$, i.e. (\ref{eq:strength}), we have $\hat d \in \mbox{rec}(C)$ (consider e.g. $\beta_{\hat k}=|y_{\hat k}|$). 

Now suppose there exists some $\bar k \in K_y$ such that $y_{\bar k} > y_{\bar k}^*$. Then there must also exist some $\bar m$ such that $\lambda^*_{\bar m}r^{\bar m} - y_{\bar k} r^{\bar k} \notin \mbox{rec}(C)$. We can easily check that $\hat{d} = \lambda^*_{\bar m}r^{\bar m} - y_{\bar k} r^{\bar k}$ satisfies \eqref{eq:sys4}-\eqref{eq:sys6}. Therefore, we have at least one extreme ray $\hat d$ outside of $\mbox{rec}(C)$.\\


\noindent\emph{Characterization of $W$:}\\
We have shown that the extreme rays of $\mbox{cl}(W)$ are in $\mbox{rec}(C)$ iff $y\leq y^*$.  If there is an extreme ray of $\mbox{cl}(W)$ outside $\mbox{rec}(C)$ then, since $C$ is closed and convex, $W\not\subseteq \mbox{int}(C)$.  Hence the contrapositive is proved ($\implies$).

Suppose $y\leq y^*$.  The extreme points of $\mbox{cl}(W)$ are the finite intersection points of $P'$ with $C$.  Thus every point in $\mbox{conv}(\mbox{ext}(\mbox{cl}(W)))$ may be written as $\bar x + x_e$ for some $x_e \in \mbox{conv}(\{\lambda_m^*r^m | m\in M\})$. Let $w$ be a point in $W$.  By convexity we may write $w$ in terms of the extreme points and rays of $\mbox{cl}(W)$: $w=\bar x + x_e + d_r + \alpha d_k$, for some $d_r \in \mbox{rec}(W)\subseteq \mbox{rec}(C)$, $\alpha\geq 0$, and $d_k\in \mbox{cone}(\{r^{k}|k \in K_y\})\subseteq \mbox{rec}(C)$. Now by definition of $W$ we have that $w$ strictly satisfies \Cref{eq:sys3}, and so $\alpha d_k>0$. Furthermore, by construction of $y^*$, for $\gamma \geq \max_{k\in K_y}\{-y_k^*\}$ we have that $\lambda_m^*r^m + \gamma d_k \in \mbox{rec}(C) \, \forall m\in M$ and thus $x_e + \gamma d_k+d_r \in \mbox{rec}(C)$.  Since $\bar x \in \mbox{int}(C)$, we have $x(\gamma):=\bar x + x_e + \gamma d_k + d_r \in \mbox{int}(C)$. So $w$ lies between $\bar x + x_e+d_r\in C$ and the interior point $x(\gamma)\in\mbox{int}(C)$, so $w\in\mbox{int}(C)$. Since our choice of $w$ is arbitrary, we have $W \subseteq \mbox{int}(C)$.  Hence the converse is true ($\impliedby$). 
 \qed
\end{proof}

It can also be shown (formal proofs are omitted for brevity) that if $|M|=1$, then $V_y \cap P' = \mbox{conv}(P'\setminus \mbox{int}(C))$. When $|M|>1$, more than one cut is needed, in general, to describe this convex hull.

\subsection{Summary of Intersection Cut Generation}
\label{subsec:applyint}
Summarizing the above discussion, an intersection cut for $P$ requires:

\begin{enumerate}
  \item A simplicial conic relaxation $P'\supseteq P$ with some apex $\bar x \notin S$.
  \item An $S$-free set $C$ containing $\bar x$ in its interior.
  \item For each extreme ray of $P'$, either the intersection with the boundary of $C$, or else proof that the ray is contained entirely in $C$.
  \item In case a ray in contained entirely in $C$, access to $\mbox{rec}(C)$ is needed for the strengthening procedure.
\end{enumerate}

Step 1 is satisfied if one selects $n$ linearly independent inequalities from the description of $P$, i.e. a basis of $P$. A standard approach is to select an optimal basis obtained by solving the corresponding LP relaxation over $P$. This is not necessarily the best choice: for example, infeasible basic solutions can also be used \cite{balas2008optimizing,fischetti2007optimizing}. 
Step 2, in the case of polynomial optimization, is the focus of \Cref{sec:poly}. 
Step 3 and 4, in the case of polynomial optimization, is addressed in \Cref{sec:polycuts}.  Note that finding an intersection point precisely on the boundary is not necessary.  A simple way in practice to ensure numerical `safety' of the cut is to use a point between $\bar x$ and the intersection point with $C$. Computing such a point is computationally straightforward provided one can quickly determine membership in $C$.  

\section{Outer-Product-Free Sets}
\label{sec:poly}
We now turn to polynomial optimization ($\mathbf{PO}$), with the aim of generating strong polyhedral relaxations using the strengthened intersection cut framework described above. Our approach to $\mathbf{PO}$ leverages the moment/sum-of-squares approach to polynomial optimization (see \cite{lasserre2001global,laurent2009sums}) from where a definition of the feasible set as $S\cap P$ is naturally obtained. 

Let $m_r=[1,x_1,\ldots ,x_n,\allowbreak x_1x_2,...,\allowbreak x_n^2,\ldots,x_n^r]$ be a vector of all monomials up to degree $r$.  Any polynomial may be written in the form $p_i(x) = m_r^T A_i m_r$ (provided $r$ is sufficiently large), where $A_i$ is a symmetric matrix derived from coefficients of $p_i$.  
We can apply this transformation to \textbf{PO} to obtain a lifted representation $\mathbf{LPO}$:
\begin{subequations}
\begin{alignat}{2}
\min \ & \langle A_0,X\rangle \nonumber\\
(\mathbf{LPO})\ \text{s.t. }&\langle A_i,X\rangle \leq b_i, &\ i=1,...,m, \label{eq:poly}\\
&X=m_rm_r^T. \label{eq:outer}
\end{alignat}
\end{subequations}

Denote $n_r:= \binom{n+r}{r}$, i.e. the length of $m_r$. Here $A_i\in \mathbb{S}^{n_r\times n_r}$ are symmetric real matrices of data, and $X\in \mathbb{S}^{n_r\times n_r}$ is a symmetric real matrix of decision variables. The problem has linear objective function, linear constraints~(\ref{eq:poly}), and nonlinear constraints~(\ref{eq:outer}). One can replace the moment matrix condition $X=m_rm_r^T$ with the equivalent conditions of $X\succeq 0$, $\mbox{rank}(X)=1$ and linear consistency constraints enforcing that entries from $X$ representing the same monomial terms are equal. 

Dropping the nonconvex rank one constraint yields the standard SDP relaxation \cite{shor_quadratic_1987}.  The relaxation is said to be exact when there is an optimal solution where $\mbox{rank}(X)=1$ since the solution can be factorized to obtain an optimal solution vector for $\mathbf{PO}$. In special cases (e.g. \cite{lovasz1991cones,lasserre2001global,laurent2009sums}) the relaxation is guaranteed to be exact for sufficiently large $r$.  

Note that (as presented) there is a combinatorial explosion in the size of \textbf{LPO} (hence the size of any associated relaxation) with respect to $r$. This is not a critical issue for our purposes, as several remedies are available that can accommodate our cuts such as: projection, partial lifting, and lower-degree reformulation with auxiliary variables. Such procedures have associated tradeoffs between relaxation quality, cut quality, size, and speed, and we leave detailed exploration of this outside the scope of the paper in order to focus on cut generation for a given choice of $r$.

The feasible region of \textbf{LPO} has a natural description as an intersection of a polyhedron $P_{OP}$, that corresponds to linear constraints~(\ref{eq:poly}) together with consistency constraints, and the following closed set,
\[S_{OP}:=\{X \in \mathbb{S}^{n_r\times n_r}: X=xx^T, x\in \mathbb{R}^{n_r}\}.\] We shall refer to matrices in $S_{OP}$ as (real) symmetric outer-products (with nonnegative diagonal). Accordingly, we shall study sets that are \emph{outer-product-free} (OPF): closed, convex sets in $\mathbb{S}^{n_r\times n_r}$ with interiors that do not intersect with $S_{OP}$. 


In what follows, we suppose we have an extreme point $\bar X \in P_{OP}$ with spectral decomposition $\bar X := \sum_{i=1}^{n_r} \lambda_i d_id_i^T$ and ordering $\lambda_1\geq ... \geq \lambda_{n_r}$. We seek to separate $\bar X$ if it is not in $S_{OP}$. 

\subsection{Oracle-Based Outer-Product-Free Sets}
\label{sec:ellipsoids}
Recall from \Cref{sec:ballcut} that our oracle-based cut requires the distance to $S$, 
which in the case of $\mathbf{LPO}$ corresponds to the distance to the 
nearest positive semidefinite matrix with rank at most one.  This distance can be obtained as a special case of the following positive semidefinite matrix approximation problem, given an integer $q > 0$:
\begin{equation}
(\mathbf{PMA})\ \min_Y \ \{ \|\bar X-Y\|  \, : \, \text{rank}(Y)\leq q,\, Y\succeq 0\}.
\end{equation}
Here $\|\cdot\|$ is a unitarily invariant matrix norm such as the Frobenius norm, $\|\cdot\|_F$. Dax \cite{dax2014low} proves the following:

\begin{theorem} [Dax's Theorem]
\label{thm:dax}
Let $k$ be the number of nonnegative eigenvalues of $\bar X$. For $1 \leq q\leq n-1$, an optimal solution to \textbf{PMA} is given by $Y = \sum_{i=1}^{\min\{k,q\}}\lambda_id_id_i^T$.  
\end{theorem}

This can be considered an extension of an earlier result by Higham \cite{higham1988computing} for $q=n$.  When $\bar X$ is not negative semidefinite, the solution from Dax's theorem coincides with the Eckart-Young-Mirsky \cite{mirsky1960symmetric,eckart1936approximation} solution to \textbf{PMA} without the positive semidefinite constraint.  The optimal positive semidefinite approximant allows us to construct an outer-product-free ball:

\[\mathcal{B}_{\text{oracle}}(\bar X):= \left\{\begin{array}{ll}
\mathcal{B}(\bar X,\|\bar X\|_F ), & \text{if} \ \bar X \text{ is NSD,}\\
\mathcal{B}(\bar X,\|\sum_{i=2}^n\lambda_id_id_i^T\|_F), & \text{otherwise.}\end{array}\right.\]

\begin{corollary}
\label{cor:opball}
$\mathcal{B}_{\textnormal{oracle}}(\bar X)$ is outer-product-free. 
\end{corollary}
\begin{proof}
Setting $q=1$ in Dax's Theorem, we see that the nearest symmetric outer product is either $\lambda_1d_1d_1^T$ if $\lambda_1>0$, or else the zeros matrix. \qed
\end{proof}

In the generic construction the oracle ball is centered around $\bar X$ since no further structure is assumed upon $S$ when using an oracle. However, for \textbf{LPO} we can in certain cases use a simple geometric construction to obtain a larger ball containing the original one, as follows.
Consider a ball $\mathcal{B}(X, r)$.  Let $s > 0$ and suppose $Q$ is in the
  boundary of the ball.  We construct the ``shifted'' ball $\mathcal{B}(Q+(s/r)(X-Q),s)$.   This ball has radius $s$ and its center is located on the ray through $X$ emanating
  from $Q$.

\begin{remark}
\label{rem:ball}
If $s>r$ then the shifted ball contains the original ball. Algebraically we may write that for any $s>r$ we have 
\[\mathcal{B}((s/r)X+(1-s/r)Q,s)\supset \mathcal{B}(X,r) \ \forall Q \in\mathbb{S}^{n\times n},\]
or $\mathcal{B}(Q+(s/r)(X-Q),s)\supset \mathcal{B}(X,r) \ \forall Q$ in the boundary of $\mathcal{B}(X,r)$. \end{remark}

Hence we can design a shifted oracle ball by choosing a point on the  boundary of  $\mathcal{B}_\text{oracle}$ and proceeding accordingly.  Let us use the nearest symmetric outer product as the boundary point in our construction:

\[\mathcal{B}_{\text{shift}}(\bar X,s):= \left\{\begin{array}{ll}
\mathcal{B}(s\bar X/\|\bar X\|_F,s), & \text{if} \ \bar X \text{ is NSD,}\\
\mathcal{B}\left(\lambda_1d_1d_1^T \, +\, \frac{s}{\|\bar X-\lambda_1d_1d_1^T\|_F}(\bar X - \lambda_1d_1d_1^T)\, , \, s\right), & \text{otherwise.}\end{array}\right.\]

\begin{proposition}
\label{prop:shiftprop}  
Suppose $\bar X \notin S_{OP}$.  If $\lambda_2\leq 0$ then $\mathcal{B}_{\textnormal{shift}}(\bar X,\|\bar X\|_F+\epsilon)$ is outer-product-free and strictly contains $\mathcal{B}_{\textnormal{oracle}}(\bar X)$ for $\epsilon>0$. If $0<\lambda_2<\lambda_1$, then for $ \left\|\sum_{i=2}^n\lambda_id_id_i^T\right\|_F<s \leq  \frac{\lambda_1}{\lambda_2} \left\|\sum_{i=2}^n\lambda_id_id_i^T\right\|_F$, $\mathcal{B}_{\textnormal{shift}}(\bar X,s)$ is outer-product-free and strictly contains $\mathcal{B}_{\textnormal{oracle}}(\bar X)$.  
\end{proposition}
\begin{proof}
Strict containment is assured by construction, so it suffices to show that $\mathcal{B}_\text{shift}$ is outer-product-free. First suppose $\bar X$ is negative semidefinite.  Then,
\begin{alignat*}{2}
&\mathcal{B}_\text{shift}(\bar X,\|\bar X\|_F+\epsilon) &\ =\ &\mathcal{B}((\|\bar X\|_F+\epsilon)\bar X/\|\bar X\|_F,\|\bar X\|_F+\epsilon),\\
&&\ =\ &\mathcal{B}\left((1+ \frac{\epsilon}{\|\bar X\|_F})\bar X,\|\bar X\|_F+\epsilon \right).
\end{alignat*}
The matrix $(1+\epsilon/\|\bar X\|_F)\bar X$ is negative semidefinite due to our negative semidefinite assumption on $\bar X$, so from Dax's theorem we know the nearest outer product is the all zeros matrix.  Hence for $\mathcal{B}_{\text{shift}}$ to be outer-product-free, its radius can be no more than the Frobenius norm of its center, $(1+\epsilon/\|\bar X\|_F)\bar X$, which by observation is indeed the case.

Now suppose $\bar X$ has at least one positive eigenvalue. Then for $s > 0$
\begin{alignat*}{2}
&\mathcal{B}_\text{shift}(\bar X,s) &\ =\ &\mathcal{B}\left(\lambda_1d_1d_1^T+ \frac{s}{\|\bar X-\lambda_1d_1d_1^T\|_F}(\bar X-\lambda_1d_1d_1^T),s \right),\\
&&\ =\ &\mathcal{B}\left(\lambda_1d_1d_1^T+\frac{s}{\|\sum_{i=2}^n\lambda_id_id_i^T\|_F}\sum_{i=2}^n\lambda_id_id_i^T,s\right).
\end{alignat*}  

If $\lambda_2\leq 0$ then the nearest outer product to the center of the
shifted ball, by Dax's theorem, is $\lambda_1d_1d_1^T$.  Thus, the maximum radius
of an outer-product-free ball centered at
  $$\lambda_1d_1d_1^T+ \frac{s}{\|\sum_{i=2}^n\lambda_id_id_i^T\|_F}\sum_{i=2}^n\lambda_id_id_i^T \ $$ 
  is $\|(s/\|\sum_{i=2}^n\lambda_id_id_i^T\|_F)\sum_{i=2}^n\lambda_id_id_i^T\|_F = s$, and so $\mathcal{B}_\text{shift}(\bar X, \|\bar X\|_F+\epsilon)$ is outer-product-free for all $\epsilon>0$.

If $0<\lambda_2<\lambda_1$, then again by Dax's theorem the nearest outer product to the center of the shifted ball is $\lambda_1d_1d_1^T$ iff 
\begin{alignat*}{2}
&\lambda_1 \, \geq \, \frac{s}{\|\sum_{i=2}^n\lambda_id_id_i^T\|_F}\lambda_2, \quad \text{i.e, iff} \quad s \, \leq \,  \frac{\lambda_1}{\lambda_2} \left\|\sum_{i=2}^n\lambda_id_id_i^T\right\|_F.
\end{alignat*}
This gives us a maximum radius of 
\[\left\| \frac{s}{\|\sum_{i=2}^n\lambda_id_id_i^T\|_F}\sum_{i=2}^n\lambda_id_id_i^T\right\|_F = s.\] \qed
\end{proof}

These results provide our first outer-product-free sets. Nonetheless, there is no guarantee they will be maximal. We now turn to characterizing and finding maximal outer-product-free sets.

\subsection{Maximal Outer-Product-Free Sets}
\label{subsec:maxopf2}

\subsubsection{General properties of maximal outer-product-free sets}

\begin{lemma}
\label{lem:intrayintpt}
Let $C$ be a full-dimensional convex set.  Every interior ray of $\mbox{cone}(C)$ emanating from the origin passes through the interior of $C$.
\end{lemma}
\begin{proof}
Consider an interior ray $r$ of $\mbox{cone}(C)$ emanating from the origin and suppose it contains a point $v \in C$.  If $r$ intersects the interior of $C$ we are done, hence assume it does not.  Then there exists a
hyperplane $H$ containing $r$ which supports $C$ at $v$.  Since $H$ contains $r$, it contains the origin.  Thus $H$ supports $\mbox{cone}(C)$ at $r$, which contradicts the assumption that $r$ is interior to $\mbox{cone}(C)$. \qed
\end{proof}

The following Theorem and Corollary provide the first building blocks towards maximality.
\begin{theorem}
\label{thm:opfcone}
Let $C\subset\mathbb{S}^{n_r\times n_r}$ be a full-dimensional outer-product-free set.  Then $\mbox{clcone}(C)$ is outer-product-free.
\end{theorem}
\begin{proof}
Suppose $\mbox{clcone}(C)$ is not outer-product-free; since it is closed and convex, then by definition of outer-product-free sets there must exist $d \in \mathbb{R}^{n_r}$ such that $dd^T$ is in its interior.  If $d$ is the zeros vector, then $\vec 0 \in \mbox{int}(\mbox{clcone}(C)) \implies \vec 0 \in \mbox{int}(C)$, which contradicts the condition that $C$ be outer-product-free.  Otherwise the ray $r^0$ emanating from the origin with nonzero direction $dd^T$ is entirely contained in and hence is an interior ray of $\mbox{clcone}(C)$. By convexity, the interior of $\mbox{cone}(C)$ is the same as the interior of its closure, so $r^0$ is also an interior ray of $\mbox{cone}(C)$.  By \Cref{lem:intrayintpt}, $r^0$ passes through the interior of $C$.  But every point along $r^0$ is a symmetric outer-product, which again implies that $C$ is not outer-product-free. \qed
\end{proof}

\begin{corollary}
\label{cor:maxcone}
Every full-dimensional maximal outer-product-free set is a convex cone.
\end{corollary}

\begin{remark}
The characterization in \Cref{thm:opfcone} allows us to expand the oracle-based outer-product-free sets in Section \ref{sec:ellipsoids}. Provided $s$ is chosen as prescribed by \Cref{prop:shiftprop}, the closure of the conic hull, $\mbox{clcone}(\mathcal{B}_{\text{shift}}(\bar X,s))$, is also outer-product-free. In the proof of \Cref{prop:shiftprop} we showed that if $\bar X$ is NSD, then for $\epsilon>0$ the outer-product-free set $\mathcal{B}_{\text{shift}}(\bar X,\|\bar X\|_F+\epsilon)$ contains the origin (the all zeros matrix) in its boundary. In this case the closure of the conic hull is a halfspace tangent to the ball at the origin. Otherwise if $\bar X$ is not NSD then we must consider two cases: either $\lambda_2$ is nonpositive or it is positive.  If $\lambda_2$ is nonpositive, then we can set $\epsilon$ to any large number, and so we have a shifted ball with arbitrarily large radius that is tangent to $\lambda_1d_1d_1^T$.  Thus in the limit as $\epsilon$ approaches infinity we obtain an outer-product-free halfspace that is tangent to $\lambda_1d_1d_1^T$ with normal parallel to the vector from $\bar X$ to $\lambda_1d_1d_1^T$.  If $\lambda_2$ is positive, then for $ \left\|\sum_{i=2}^n\lambda_id_id_i^T\right\|_F<s \leq  \frac{\lambda_1}{\lambda_2} \left\|\sum_{i=2}^n\lambda_id_id_i^T\right\|_F$, the outer-product-free ball $\mathcal{B}_{\textnormal{shift}}(\bar X,s)$ does not contain the origin. In this case the conic hull is equal to its closure \cite[Prop 1.4.7]{hiriart2012fundamentals}.
\end{remark}

We can also consider maximal outer-product-free sets via their supporting halfspaces.

\begin{definition}
A \emph{supporting halfspace} of a closed, convex set $S$ contains $S$ and its boundary is a supporting hyperplane of $S$.
\end{definition}

\begin{corollary}
\label{cor:halfspaces}
Let $C$ be a full-dimensional maximal outer-product-free set.  Every supporting halfspace of $C$ is of the form $\langle A, X \rangle \geq 0$ for some $A \in\mathbb{S}^{n_r\times n_r}$.
\end{corollary}
\begin{proof}
 From \Cref{cor:maxcone} we have that $C$ is a convex cone, and so the result follows from classic convex analysis \cite [Cor 11.7.3]{rockafellar1970convex}. \qed
\end{proof}

From \Cref{cor:halfspaces} we may thus characterize a maximal outer-product-free set as $C=\{X\in\mathbb{S}^{n_r\times n_r}| \langle A_i, X\rangle \geq 0 \ \forall i \in I\}$, where $I$ is an index set that is not necessarily finite.  It will be useful to classify the supporting halfspaces in terms of the coefficient matrix: $A$ is either positive semidefinite, negative semidefinite, or indefinite.  

We are now ready to provide our first explicit family of maximal outer-product-free sets. 
\begin{theorem}
\label{thm:psdhalf}
The halfspace $\langle A, X \rangle \geq 0$ is maximal outer-product-free iff $A$ is negative semidefinite.  
\end{theorem}
\begin{proof}
Suppose $A$ has a strictly positive eigenvalue, with corresponding eigenvector $d$.  Then  $\langle A, dd^T \rangle > 0$, and so the halfspace is not outer-product-free.

If $A$ is negative semidefinite we have $\langle A, dd^T \rangle = d^TAd \leq 0 \ \forall d \in \mathbb{R}^{n_r}$, so the halfspace is outer-product-free. For maximality, suppose the halfspace is strictly contained in another outer-product-free set $\bar C$. Then there must exist some $\bar X \in \mbox{int}(\bar C)$ such that $\langle A, \bar X \rangle <0$.  However, $\langle A, (-\bar X) \rangle >0$, and so the zeros matrix is interior to $\bar C$ since it lies between $\bar X$ and $-\bar X$.  Thus $\bar C$ cannot be outer-product-free. \qed
\end{proof}

\begin{corollary}
Let $C$ be a full-dimensional maximal outer-product-free set with supporting halfspaces $ \{ \langle A_i, X \rangle \geq 0 \, : \, i \in I \}.$ Here, the set $I$ is possibly infinite. If there exists $i \in I$ such that $A_i$ is negative semidefinite, then $C$ is exactly the halfspace $\langle A_i, X \rangle \geq 0$.
\end{corollary}
\begin{proof}
Suppose $C$ is contained in the halfspace $\langle A_i, X \rangle \geq 0$ with $A_i$ NSD.  By \Cref{thm:psdhalf} the halfspace is outer-product-free, and so $C$ is maximal only if it is the supporting halfspace itself. \qed
\end{proof}

\subsubsection{Maximal outer-product-free sets derived from $2\times 2$ submatrices}

Another important family of maximal outer-product-free sets can be obtained using the following characterization of rank-1 PSD matrices by Kocuk, Dey, and Sun \cite{kocuk2017matrix}:

\begin{proposition}[KDS Proposition]
\label{prop:KDS}
A nonzero, Hermitian matrix $X$ is positive semidefinite and has rank one iff all the $2\times 2$ minors of $X$ are zero and the diagonal elements of $X$ are nonnegative.\qed
\end{proposition}

Denote the entries of a $2\times 2$ submatrix of $X$ from some rows $i_1 < i_2$ and columns $j_1< j_2$ as $X_{[[i_1, i_2], [j_1, j_2]]}:= \left[\begin{array}{cc}
a & b\\
c & d\end{array}\right]$.

\begin{lemma}
\label{lemma:opfcones}
Let $\lambda \in \mathbb{R}^2$ with $\| \lambda \|_2 = 1$.  Then each of the following describes an outer-product-free set:
\begin{subequations}\label{eq:22_1}
\begin{alignat}{2}
\lambda_1(a+d)+\lambda_2(b-c) \geq    \|b+c,a-d\|_2, \label{eq:22_1a}\\
\lambda_1(b+c)+\lambda_2(a-d) \geq   \|a+d,b-c\|_2. \label{eq:22_1b}
\end{alignat}
\end{subequations}
\end{lemma}
\begin{proof}
First consider \eqref{eq:22_1a}. From \Cref{prop:KDS} we have the necessary condition for an outer product that $ad = bc$.  Hence the set of symmetric matrices satisfying $ad \geq bc$ contains no outer product in its interior. Note that
\begin{alignat*}{2}
ad \geq bc &\iff (a+d)^2-(a-d)^2 \geq (b+c)^2-(b-c)^2\\
&\iff \|a+d,b-c\|_2 \geq \|b+c,a-d\|_2.
\end{alignat*}

Clearly $  \|a+d,b-c\|_2 \ge \lambda_1(a+d)+\lambda_2(b-c)$, and so \eqref{eq:22_1a} describes
a subset of the set satisfying $ad \geq bc$, thus, it is outer-product-free. \eqref{eq:22_1b} follows in the same way, replacing $a$ for $b$ and $c$ for $d$ in the proof for \eqref{eq:22_1a}. \qed
\end{proof}

The following Theorem provides an extensive list of \emph{maximal} outer-product-free sets that can be obtained from Lemma \ref{lemma:opfcones}.
\begin{theorem}
\label{thm:22thm}
(\ref{eq:22_1a}) describes a \emph{maximal} outer-product-free set if

i) $\lambda_1=1,\lambda_2=0$, and neither $b$ nor $c$ are diagonal entries;

ii) $\lambda_1=0,\lambda_2=1$, and $b$ is a diagonal entry;

iii) $\lambda_1=0,\lambda_2=-1$, and $c$ is a diagonal entry;

iv) $\lambda_1^2+\lambda_2^2=1$, and none of $a,b,c,d$ are diagonal entries.\\

Similarly, (\ref{eq:22_1b}) describes a \emph{maximal} outer-product-free set if

v) $\lambda_1=1,\lambda_2=0$, and either $b$ or $c$ is a diagonal entry;

vi) $\lambda_1=0,\lambda_2=1$, and $a$ but not $d$ is a diagonal entry;

vii) $\lambda_1=0,\lambda_2=-1$, and $d$ but not $a$ is a diagonal entry;

viii) $\lambda_1^2+\lambda_2^2=1$, and none of $a,b,c,d$ are diagonal entries.

\end{theorem}

\begin{proof}
The outer-product-free property is given by \Cref{lemma:opfcones}, so maximality remains. Let $C$ be a set described by (\ref{eq:22_1a}) or (\ref{eq:22_1b}). It suffices to construct, for \emph{every} symmetric matrix $\bar X \not\in C$, $Z\coloneqq zz^T$ such that $Z \in \mbox{int} (\mbox{conv}(C\cup\bar X))$. As $C$ is a cone, this is equivalent to showing $Z-\bar X\in \mbox{int}(C)$, and so $\mbox{conv}(C\cup\bar X)$ is not outer-product-free (irrespective of $\bar X$). This would imply that $C$ cannot be contained in a larger $S$-free set.  

Denote the submatrices of $\bar X, Z$:
\[\bar X_{[[i_1, i_2], [j_1, j_2]]}\coloneqq  \left[\begin{array}{cc}
\bar a & \bar b\\
\bar c & \bar d\end{array}\right], \ Z_{[[i_1, i_2], [j_1, j_2]]}\coloneqq  \left[\begin{array}{cc}
 a_Z &  b_Z\\
c_Z & d_Z\end{array}\right].\]
Furthermore, for convenience let us define the following:
\[
\bar p \coloneqq  (\bar a + \bar d)/2,\, \bar q \coloneqq  (\bar a - \bar d)/2,\,
\bar r \coloneqq  (\bar b + \bar c)/2,\, \bar s \coloneqq  (\bar b - \bar c)/2.\]

\noindent \emph{Construction for (\ref{eq:22_1a})}: Suppose $\bar{X}$ violates (\ref{eq:22_1a}). We propose the following:
\begin{equation} \label{eq:defABCDfirst}
\left[\begin{array}{cc}
 a_Z &  b_Z\\
c_Z & d_Z\end{array}\right] = \left[\begin{array}{cc}
 \bar q + \lambda_1\|\bar q, \bar  r\|_2 & \ \   \bar  r+ \lambda_2\|\bar  q, \bar r\|_2\\
\bar  r - \lambda_2\|\bar  q,\bar r\|_2 & \ \ -\bar q + \lambda_1\|\bar q,\bar r\|_2\end{array}\right],
\end{equation}
\begin{align*}
\Longrightarrow \, \lambda_1(\bar a+\bar d)/2+\lambda_2(\bar b-\bar c)/2 \, &< \, \|(\bar b+\bar c)/2,(\bar a-\bar d)/2\|_2 \\ 
&= \, \lambda_1(a_Z+ d_Z)/2+\lambda_2(b_Z-c_Z)/2
\end{align*}
where the last equality follows from $\lambda_1^2 + \lambda_2^2 = 1$. This implies
\[ \lambda_1((a_Z-\bar a)+ (d_Z- \bar d)) +\lambda_2((b_Z-\bar b)-(c_Z-\bar c)) >  0 \]
and since \( \|(b_Z-\bar b)+(c_Z-\bar c),(a_Z-\bar a)-(d_Z-\bar d)\|_2 = 0 \),
we conclude $Z-\bar{X} \in \mbox{int}(C)$.\\

\noindent \emph{Construction for (\ref{eq:22_1b})}: If $\bar{X}$ violates (\ref{eq:22_1b}), we use the following construction:
\begin{equation}\label{eq:defABCDsecond}
\left[\begin{array}{cc}
 a_Z &  b_Z\\
c_Z & d_Z\end{array}\right] = \left[\begin{array}{cc}
 \bar p + \lambda_2\|\bar p,\bar  s\|_2 & \ \   \bar  s+ \lambda_1\|\bar  p,\bar  s\|_2\\
-\bar  s + \lambda_1\|\bar  p,\bar s\|_2 & \ \ \bar p - \lambda_2\|\bar p,\bar  s\|_2\end{array}\right],
\end{equation}

\begin{align*}
\Longrightarrow \lambda_1(\bar b+\bar c)/2+\lambda_2(\bar a-\bar d)/2  \, & < \, \|(\bar a+\bar d)/2,(\bar b-\bar c)/2\|_2 \\ &= \,  \lambda_1(b_Z+ c_Z)/2+\lambda_2(a_Z-d_Z)/2, \nonumber\\
\Longrightarrow \, \lambda_1((b_Z-\bar b)+(c_Z-\bar c))  &+\lambda_2((a_Z-\bar a)- (d_Z- \bar d))  > 0. \nonumber
\end{align*}
We conclude $Z-\bar X \in \mbox{int}(C)$ as before, since \( \|(a_Z-\bar a)+(d_Z-\bar d),(b_Z-\bar b)-(c_Z-\bar c)\|_2 = 0.\)
It remains to set the other entries of $Z$ and to show it is an outer product.

\begin{claim}
For each condition (i)-(viii), $a_Zd_Z=b_Zc_Z$ and all diagonal elements among $a_Z,b_Z,c_Z,d_Z$ are nonnegative.
\end{claim}
\begin{claimproof}
First consider conditions (i)-(iv). By construction of \eqref{eq:defABCDfirst}:
\begin{alignat*}{2}
&a_Zd_Z\, &=&\, -\bar q^2 + \lambda_1^2\|\bar q,\bar  r\|^2_2 \, = \, \bar r^2-\lambda_2^2\|\bar q,\bar r\|_2^2\,
= \, b_Zc_Z.
\end{alignat*}

The second equality is derived from the following identity:
\begin{alignat*}{2}
&\|\bar q,\bar r\|_2^2 = \bar q^2+\bar r^2
\iff\, -\bar q^2 + \lambda_1^2\|\bar q,\bar r\|_2^2 = \bar r^2-\lambda_2^2\|\bar q,\bar r\|_2^2.
\end{alignat*}
Now we only need to prove that diagonal elements of Z are nonnegative. To see this, first notice that in case (i) only $a_Z$ or $d_Z$ can be diagonal elements, and they are both nonnegative because of $\|\bar q, \bar r\|_2 \geq \max\{|\bar q|,|\bar r|\}$. In case (ii) and (iii), $b_Z$ and $c_Z$, respectively, can be diagonal elements and, again in view of $\|\bar q, \bar r\|_2 \geq \max\{|\bar q|,|\bar r|\}$ these are nonnegative. In case (iv) there are no diagonal elements so there is nothing to prove.

Now, for conditions (v)-(viii), 
\begin{alignat*}{2}
&a_Zd_Z\, &=&\, \bar p^2 - \lambda_2^2\|\bar p,\bar  s\|^2_2 \, = \, -\bar s^2+\lambda_1^2\|\bar p,\bar s\|_2^2 
\, = \, b_Zc_Z.
\end{alignat*}
The second equality is derived from the following identity:
\begin{align*}
\|\bar p,\bar s\|_2^2 = \bar p^2+\bar s^2
& \iff \, -\bar s^2 + \lambda_1^2\|\bar p,\bar s\|_2^2 = \bar p^2-\lambda_2^2\|\bar p,\bar s\|_2^2.
\end{align*}

Nonnegativity of diagonal elements follows from the same argument as before, by using the fact that $\|\bar p, \bar s\|_2 \geq \max\{|\bar p|,|\bar s|\}$.
\end{claimproof}
\vskip .5cm
To maintain symmetry we set $Z_{i_1,j_1}=Z_{j_1,i_1}$, $Z_{i_1,j_2}=Z_{j_2,i_1}, Z_{i_2,j_1}=Z_{j_1,i_2}$, $Z_{i_2,j_2}=Z_{j_2,i_2}$.  Now denote $\ell = [i_1,i_2,j_1,j_2]$. If $a_Z=b_Z=c_Z=d_Z=0$, then we simply set all other entries of $Z$ equal to zero and so $Z$ is the outer product of the vector of zeroes.  Otherwise, consider the following cases.\\

\noindent \emph{Case 1: $\ell$ has 4 unique entries}. Suppose w.l.o.g we have an upper-triangular entry $(i_1<i_2<j_1<j_2)$ and furthermore suppose that $b_Z$ is nonzero.  Set 

\[Z_{\ell}\coloneqq  \left[\begin{array}{cccc}
1 & d_Z/b_Z & a_Z &  b_Z \\
d_Z/b_Z& d_Z^2/b_Z^2& c_Z & d_Z\\
a_Z & c_Z & a_Z^2&a_Zb_Z\\
b_Z & d_Z & a_Zb_Z&b_Z^2
\end{array}\right], z_{\ell} := [\begin{array}{cccc}1&  d_Z/b_Z &  a_Z &  b_Z \end{array}],\]
and all remaining entries of $Z$ (and $z$) to zero. Recall that $a_Zd_Z=b_Zc_Z$, and so $Z=zz^T$. Other orderings of indices or the use of a different nonzero entry is handled by relabeling/rearranging column/row order.\\

\noindent \emph{Case 2: $\ell$ has three unique entries}. Then, exactly one of $a_Z,b_Z,c_Z,d_Z$ is a diagonal entry, and so cases (i)-(iii), (v)-(vii) apply.  For cases (i) and (vi)-(vii), where either  $a_Z$ or $d_Z$ is on the diagonal, by construction $|b_Z|=|c_Z|$. As $a_Zd_Z=b_Zc_Z$, we have $b_Z=c_Z=0$ iff exactly one of $a_Z$ or $d_Z$ is zero.  Likewise, for cases (ii)-(iii) and (v), where either $b_Z$ or $c_Z$ is a diagonal element, then $|a_Z|=|d_Z|$ and so $a_Z=d_Z=0$ iff exactly one of $b_Z$ or $c_Z$ are zero. 

Suppose $a_Z$ is a nonzero diagonal entry.  We propose:
\[Z_{\ell'} \coloneqq \left[\begin{array}{ccc}
a_Z & b_Z &  c_Z \\
b_Z& b_Z^2/a_Z&  d_Z\\
c_Z& d_Z& c_Z^2/a_Z 
\end{array}\right]\]
where $\ell'$ are the unique entries of $\ell$, and all other entries of $Z$ are set to zero. If $a_Z=0$ and is on the diagonal, then we replace $b_Z^2/a_Z$ and $c_Z^2/a_Z$ with $|d_Z|$. If $b_Z,c_Z$ or $d_Z$ is on the diagonal, we use the same construction but with relabeling/rearranging column/row order.\\

\noindent \emph{Case 3: $\ell$ has two unique entries}. All remaining entries of $Z$ are set to zero.\\

In all cases, all diagonal entries of $Z$ are nonnegative and all $2\times 2$ minors are zero; by \Cref{prop:KDS}, $Z$ is an outer product. \qed
\end{proof}

Additionally, we show a choice of $\lambda$ in \Cref{lemma:opfcones} which ensures separation.

\begin{lemma} \label{lemma:22always}
Consider a matrix $\bar X$ and a $2\times 2$ submatrix
\[\bar X_{[[i_1, i_2], [j_1, j_2]]}\coloneqq  \left[\begin{array}{cc}
\bar a & \bar b\\
\bar c & \bar d\end{array}\right] \]
such that $\bar{a} \bar{d} \neq \bar{b} \bar{c}$. Then,
\begin{itemize}
\item If $\bar{a} \bar{d} > \bar{b} \bar{c}$, $\bar{X}$ is in the interior of the cone defined by \eqref{eq:22_1a} with
\begin{subequations}
\begin{equation} \label{eq:firstlambdas}
\lambda_1  = \frac{\bar{a} + \bar{d}}{\| \bar{a} + \bar{d} , \bar{b} - \bar{c} \|_2}, \quad \lambda_2  = \frac{\bar{b} - \bar{c}}{\| \bar{a} + \bar{d} , \bar{b} - \bar{c} \|_2}.
\end{equation}
\item If $\bar{a} \bar{d} < \bar{b} \bar{c}$, $\bar{X}$ is in the interior of the cone defined by \eqref{eq:22_1b} with
\begin{equation} \label{eq:secondlambdas}
\lambda_1  = \frac{\bar{b} + \bar{c}}{\| \bar{b} + \bar{c} , \bar{a} - \bar{d} \|_2}, \quad \lambda_2  = \frac{\bar{a} - \bar{d}}{\| \bar{b} + \bar{c} , \bar{a} - \bar{d} \|_2}. 
\end{equation}
\end{subequations}
\end{itemize}
\end{lemma}
\begin{proof}
Let us first consider the case $\bar{a} \bar{d} > \bar{b} \bar{c}$, and $\lambda$ defined as in \eqref{eq:firstlambdas}, then
\[
    \lambda_1(\bar{a}+ \bar{d})+\lambda_2(\bar{b}-\bar{c}) - \|\bar{b}+\bar{c},\bar{a}- \bar{d}\|_2  = \| \bar{a} + \bar{d} , \bar{b} - \bar{c} \|_2 - \|\bar{b}+\bar{c},\bar{a}- \bar{d}\|_2,
\]
and since 
\[
\| \bar{a} + \bar{d} , \bar{b} - \bar{c} \|^2_2 - \|\bar{b}+\bar{c},\bar{a}- \bar{d}\|^2_2  = 4 \bar{a} \bar{d} - 4 \bar{b} \bar{c} > 0,
\]
we conclude $\bar{X}$ is in the interior of the cone defined by \eqref{eq:22_1a}. When $\bar{a} \bar{d} < \bar{b} \bar{c}$ we obtain the second case which involves \eqref{eq:22_1b} the same way. 
\qed
\end{proof}

\begin{remark}
Note that in the proof above the $\lambda$ choices are, in a sense, the best possible. The following optimization problem 
\[ \max\{ \lambda_1 x + \lambda_2 y \, : \, \lambda_1^2 + \lambda_2^2 = 1 \} \]
has the optimal solution \(\lambda_1 = x/\|x,y\|_2,\, \lambda_2 = y/\|x,y\|_2\). This proves that $\lambda$ defined in \eqref{eq:firstlambdas} maximizes the difference between both sides of inequality \eqref{eq:22_1a}. The same holds for $\lambda$ defined in \eqref{eq:secondlambdas} with respect to \eqref{eq:22_1b}. Note that this is ``best'' in a violation sense, and may not translate to finding the \emph{deepest} cut.
\end{remark}

\subsubsection{All maximal outer-product-free sets when $n_r = 2$}

As a consequence of \Cref{thm:22thm} we can characterize \emph{all} outer-product-free sets in $\mathbb{S}^{2\times 2}$. The next corollary follows from \Cref{thm:22thm}.

\begin{corollary}
\label{cor:psdmax}
$\mathbb{S}^{n_r\times n_r}_+$ is maximal outer-product-free iff $n_r\leq 2$.
\end{corollary}
\begin{proof}
Let $a,d$ be diagonal entries so $b=c$ by symmetry.  Then case (i) in \Cref{thm:22thm} describes the maximal outer-product-free set $\bar C$ given by $a+d\geq \|2b,a-d\|_2$, i.e. a $2\times 2$ principal submatrix is PSD.  If $n_r>2$, $\bar C$ strictly contains $\mathbb{S}^{n_r\times n_r}_+$.  Hence, $\mathbb{S}^{n_r\times n_r}_+$ cannot be maximal. 

If $n_r=2$, then $\bar C = \mathbb{S}^{n_r\times n_r}_+$. For $n_r <2$, maximality is trivial.
\end{proof}

\begin{lemma}
\label{lem:psdunique}
In $\mathbb{S}^{2\times 2}$ the cone of positive semidefinite matrices is the unique maximal outer-product-free set containing at least one positive semidefinite matrix in its interior.  
\end{lemma}
\begin{proof}
From \Cref{cor:psdmax} we have that the cone of PSD matrices is maximal for $n_r = 2$.  Hence, if there exists a maximal outer-product-free set containing a PSD matrix in its interior, it consequently contains in its interior a boundary point of $\mathbb{S}_+^{2\times 2}$  ---otherwise, it is a subset of the PSD cone.  However, every boundary point of the PSD cone has at least one zero-valued eigenvalue; it follows that for $n_r = 2$ every such point is a symmetric outer product. \qed
\end{proof}

\begin{theorem}
\label{thm:smallmaxfree}
In $\mathbb{S}^{2\times 2}$ every full-dimensional maximal outer-product-free set is either the cone of positive semidefinite matrices or a halfspace of the form $\langle A,X\rangle \geq 0$, where $A$ is a symmetric negative semidefinite matrix.
\end{theorem}
\begin{proof}
From \Cref{lem:psdunique}, we have that every maximal outer-product-free set is either the cone of positive semidefinite matrices or it does not contain a PSD matrix in its interior.  Now suppose $C \in \mathbb{S}^{n_r\times n_r}$ is a maximal outer-product-free set that is not the cone of positive semidefinite matrices. $C$ is thus a closed, convex set with interior that does not intersect with $\mathbb{S}_+^{2\times 2}$.  Then by the separating hyperplane theorem there exists a supporting hyperplane of $\mathbb{S}^{2\times 2}_+$, which by \Cref{cor:halfspaces} and \Cref{lem:psdunique} is of the form $\langle A,X\rangle = 0$, such that $C$ is contained in the halfspace $\langle A,X\rangle \geq 0$.  But if $A$ has a positive eigenvalue then the halfspace includes at least one PSD matrix; thus to maintain separation $A$ is necessarily negative semidefinite.  Furthermore, for any negative semidefinite $A$ the halfspace $\langle A,X\rangle \geq 0$ is outer-product-free by \Cref{thm:psdhalf} so $C$ must be the halfspace itself in order to be maximal outer-product-free. \qed
\end{proof}

\section{Step Lengths, Strengthening, and Separation}
\label{sec:polycuts}
In this section we discuss the implementation of Step 3 and 4 of the intersection cut as described in \Cref{subsec:applyint}. In particular, given a simplicial cone $P'$ with apex $\bar X \notin S_{OP}$, we discuss how to select appropriate outer-product-free sets among those given in \Cref{sec:poly} and how to generate the corresponding step lengths $\lambda$ in order to generate a cut for $P$ that separates $\bar X$. 

\subsection{Oracle-Based Cuts}
\label{subsec:polyocuts}
As shown in \Cref{sec:ellipsoids}, the oracle ball $\mathcal{B}_{\text{oracle}}(\bar X)$ can be used to generate a separating intersection cut for any $\bar X$ that is not a symmetric, real outer-product, and calculation of the radius and center of either ball can be done using the spectral decomposition of $\bar X$.  For $\mathcal{B}_{\text{oracle}}$, the step lengths $\lambda$ are all equal to the radius of the ball.  We can strengthen the cut further by using 
the conic extension $\mbox{clcone}(\mathcal{B}_{\text{shift}}(\bar X,s))$ (see \Cref{prop:shiftprop} and \Cref{thm:opfcone}).

If $\bar X$ is NSD, then by \Cref{prop:shiftprop} for any $\epsilon > 0$ the shifted ball $ \mathcal{B}((1+ \allowbreak \frac{\epsilon}{\|\bar X\|_F})\bar X, \allowbreak \|\bar X\|_F+\epsilon )$ is outer-product-free.  The closed conic hull of this ball (irrespective of $\epsilon$) is a halfspace tangent to the ball at the origin.  A normal vector of this halfspace is thus $\bar X/\|\bar X\|_F$, and so the equation of the halfspace is $\langle \bar X/\|\bar X\|_F,X\rangle \geq 0$.  The best possible cut from this maximal (recall \Cref{thm:psdhalf}) outer-product-free halfspace is a halfspace in the opposite direction, 
\[\langle \bar X/\|\bar X\|_F,X\rangle \leq 0.\]

Otherwise, if  $\bar X$ is not NSD, then we must determine the sign of $\lambda_2$.  If $\lambda_2$ is nonpositive, then we may use the halfspace that contains $\lambda_1d_1d_1^T$ on its boundary and that is perpendicular to the vector from $\bar X$ to $\lambda_1d_1d_1^T$, i.e. $\langle \bar X-\lambda_1d_1d_1^T,X-\lambda_1d_1d_1^T\rangle \geq 0$.  Again, the best possible cut is a halfspace in the opposite direction:
\[\langle \bar X-\lambda_1d_1d_1^T,X-\lambda_1d_1d_1^T\rangle \leq 0. \]

If $\bar X$ is not NSD and $\lambda_2$ is positive, then we may use the maximum shift prescribed by \Cref{prop:shiftprop}: $s =  \frac{\lambda_1}{\lambda_2} \left\|\sum_{i=2}^n\lambda_id_id_i^T\right\|_F$.  This gives us a shifted ball with centre 
\[X_C:=\lambda_1d_1d_1^T \, +\,\allowbreak \frac{\lambda_1}{\lambda_2}(\bar X - \lambda_1d_1d_1^T)\] 
and radius 
\[q:=\frac{\lambda_1}{\lambda_2} \left\|\bar X - \lambda_1d_1d_1^T\right\|_F.\]

The ball does not touch the origin (see proof of \Cref{prop:shiftprop}), thus $\mbox{cone}(\mathcal{B}(X_C,q))$ is outer-product-free and contains $\bar X$.  Given the $k$th extreme ray of $P'$, emanating from $\bar X$ along the direction $D^{(k)}$, we wish to determine the intersection point $Z_0:= \bar X + \lambda_k D^{(k)}$ with the boundary of $\mbox{cone}(\mathcal{B}(X_C,q))$.  First we must check if the ray is contained in the cone, i.e. if the intersection is at infinity. The scalar projection of the direction vector onto the axis of the cone is $\langle D^{(k)}, X_C \rangle/\|X_C\|_F$. If the scalar is negative, then the ray passes through the cone.  If the scalar is nonnegative, then the radius of the cone at the projected point is $r_1:= \langle D^{(k)},X_C \rangle q\allowbreak /(\|X_C\|_F\sqrt{\|X_C\|_F^2-q^2})$ (see \Cref{eq:appradfinal} in \Cref{sec:radius}). The distance from $D^{(k)}$ to the cone's axis is $d_1 := \|D^{(k)}-(\langle D^{(k)},X_C\rangle/\langle X_C,X_C\rangle )X_C\|_F$.  

If $d_1\geq r_1$ then the ray intersects with the boundary of the cone and the step length is finite.  The scalar projection of $Z_0$ onto the axis of $\mbox{cone}(\mathcal{B}(X_C,q))$ is given by $\langle Z_0,X_C \rangle/\|X_C\|_F$.  The radius of the cone at the projected point is $r_2:=\langle Z_0,X_C \rangle q\allowbreak/(\|X_C\|_F\sqrt{\|X_C\|_F^2-q^2})$. The distance from $Z_0$ to the axis is $d_2 := \|Z_0-(\langle Z_0,X_C\rangle/\langle X_C,X_C\rangle )X_C\|_F$. Intersection occurs at $d_2=r_2$, and squaring both sides yields a quadratic equation, the positive root of which yields the step length. 

Otherwise, the step length is infinite  and we may apply the strengthening procedure of \Cref{sec:strengthen}.  Let $m$ be the index of an extreme ray of $P'$ with finite intersection. Applying \Cref{eq:strength} yields
\begin{align}
 \lambda_k^{'} := \max \{ y\,|\, & \|\lambda_mD^{(m)}-yD^{(k)}- (\langle \lambda_mD^{(m)}-yD^{(k)},X_C\rangle/\langle X_C,X_C\rangle )X_C\|_F  \nonumber \\
 & \leq \langle \lambda_mD^{(m)}-yD^{(k)},X_C \rangle q/(\|X_C\|_F\sqrt{\|X_C\|_F^2-q^2}) \} \label{eq:newsteporacle}
 \end{align}

The maximum occurs when the inequality is set to equality.  Squaring both sides of the equality we again have a quadratic equation, and the step length is its (algebraically) greatest root.


\subsection{Outer-Approximation Cuts}
The maximal outer-product-free sets described by \Cref{thm:psdhalf} are halfspaces, and so the best possible cuts that can be derived from these are of the form $\langle A,X \rangle \leq 0$, where $A$ is NSD.  However, observe that if $A$ has rank $k > 1$, then the cut is of the form $\sum_{i=1}^k \lambda_i d_i^TXd_i \leq 0$, where each $\lambda_i$ is a negative eigenvalue.  Then the inequality is implied by and thus weaker than the individual inequalities of the form $\lambda_i d_i^TXd_i \leq 0$.  These individual inequalities are valid as they are necessary for the positive semidefinite condition on $X$, 
and so the halfspaces described by \Cref{thm:psdhalf} characterize the outer-approximation cuts of the SDP relaxation to \textbf{LPO}.  Therefore separation is only possible if $\bar X$ is not PSD.  We adopt a standard approach to separation (see e.g. \cite{sherali2002enhancing,saxena2011convex,qualizza2012linear,krishnan2006unifying}), using all negative eigenvectors of $\bar X$ as cut coefficient vectors.  

\subsection{$2\times 2$ Submatrix Cuts}
\label{subsec:22cone}
Consider 
the infinite family of maximal outer-product-free sets established in \Cref{thm:22thm}.  Supposing that non-negativity of $\mbox{diag}(X)$ is enforced in $P$, then from \Cref{prop:KDS} it is known $\bar X \notin S_{OP}$ implies at least one $2\times 2$ minor of $\bar X$ is nonzero. From \Cref{lemma:22always} we know there is always a cone of the form \eqref{eq:22_1a} or \eqref{eq:22_1b} which contains $\bar X$ in its interior.
Searching for an appropriate $2\times 2$ submatrix is straightforward: we can enumerate over all $2\times 2$ submatrices and check for a nonzero $2\times 2$ minor. To generate potentially several cuts from $\bar X$, we also consider maximal outer-product-free sets from (i)-(iii) and (v)-(vii) of \Cref{thm:22thm}. After we have identified a cone $C$ given by either \eqref{eq:22_1a} or \eqref{eq:22_1b} containing $\bar X$, we need to compute the corresponding step lengths. 

Fix a direction $D$, given by an extreme ray of $P'$. We first evaluate if $D \in C$, i.e. evaluating the expression \eqref{eq:22_1a} or \eqref{eq:22_1b}. If $D \not\in C$, we seek the step length $\mu \geq 0$ such that $\bar X +\mu D$ lies on the boundary of the $2\times 2$ cone $C$. Since $C$ is represented as a second-order cone, computing such $\mu$ reduces to simply computing the roots of a single-variable quadratic. On the other hand, if $D \in C$, since $C$ is a cone we have $\bar X +\mu D \in C$ for all $\mu\geq 0$. Therefore, we have an infinite step length.
In the case of infinite step length,  we can use the strengthening procedure described in \Cref{sec:strengthening}. The key step in the strengthening procedure is solving over the recession cone for \eqref{eq:strength}, which in general requires solving $|\bar M|$ optimization problems over $\mbox{rec}(C)$. However, in this case $\mbox{rec}(C) = C$, and $C$ is described as a second order cone. Therefore, solving the optimization in \eqref{eq:strength} amounts to simply finding the roots of a single-variable quadratic.



\section{Numerical Examples and Experiments}
\label{sec:exp}

\subsection{Example: Polynomial Optimization}
We provide a simple example in $\mathbb{S}^{2\times 2}$. 
Consider the following polynomial optimization problem:
\begin{alignat*}{2}
\min \ & x_1^2+x_2^2\\
\mbox{s.t. } &  -x_1^2-x_2^2 +x_{1}x_{2}\leq -2,\\
&-x_1^2-x_2^2 -x_{1}x_{2}\leq -2,\\
& -x_1^2+x_2^2-x_1x_2\leq 0.
\end{alignat*}
An $\mathbf{LPO}$ representation (ignoring linear terms) is
\begin{subequations}
\begin{alignat}{2}
\min \ & X_{11}+X_{22}\nonumber\\
\mbox{s.t. } &  -X_{11}-X_{22} +X_{12}\leq -2, \label{eq:exlin1}\\
&-X_{11}-X_{22} -X_{12}\leq -2, \label{eq:exlin2}\\
& -X_{11}+X_{22}-X_{12}\leq 0, \label{eq:exlin3}\\
&X=xx^T. \label{eq:exop}
\end{alignat}
\end{subequations}
Dropping the outer product constraint~(\ref{eq:exop}) results in a linear program ---indeed, by construction the linear constraints describe a simplicial cone.  The optimal basic solution $\bar X$ to this linear relaxation is at the apex of the simplicial cone,
\(
\bar X=\left[\begin{array}{cc}
1 & 0\\
0 & 1\end{array}\right].
\)
Taking the negative basis inverse,
\[
-\left[\begin{array}{ccc}
-1 & -1 & 1  \\
-1 & -1 & -1\\
-1 & 1 & -1 \end{array}\right]^{-1} = 
\left[\begin{array}{ccc}
0.5 & 0 & 0.5  \\
0 & 0.5 & -0.5\\
-0.5 & 0.5 & 0 \end{array}\right],
\]
we obtain the following extreme ray directions
\[
D^{(1)}=\left[\begin{array}{cc}
0.5 & -0.5\\
-0.5 & 0 \end{array}\right],
D^{(2)}=\left[\begin{array}{cc}
0 & 0.5\\
0.5 & 0.5\end{array}\right],
D^{(3)}=\left[\begin{array}{cc}
0.5 & 0\\
0 & -0.5\end{array}\right].
\]

\paragraph{$\mathbf{2\times 2}$ \textbf{Cut.}}
The procedure in \Cref{subsec:22cone} gives us the step lengths
\[\lambda_1^* = \lambda_2^* = 2\phi \approx 3.24, \lambda_3^* = 2,\]
where $\phi := \frac{1+\sqrt{5}}{2}$ is the golden ratio. Using \Cref{eq:closedformcut} we obtain the cut 
\[(0.5+\phi^{-1})X_{11}+(\phi^{-1}-0.5)X_{22}+0.5X_{12} \geq 2\phi^{-1}+1.\]
After adding this cut, the strengthened LP produces a rank one solution,
\(
\left[\begin{array}{cc}
2 & 0\\
0 & 0\end{array}\right].
\)

\paragraph{\textbf{Outer Approximation Cut.}}
As $\bar X$ is strictly positive definite, no outer approximation cut can separate it.

\paragraph{\textbf{Oracle-Based Ball Cut.}}
Both eigenvalues of $\bar X$ are equal to 1, and so the radius of the ball is 1.  Note that we must normalize the radius to obtain the step lengths, i.e. $\lambda_k = 1/\|D^{(k)}\|_F$. Hence the oracle ball cut is strictly dominated by the $2\times 2$ cut: $\lambda_1 = \lambda_2 = 2/\sqrt{3} \approx 1.15, \lambda_3 = \sqrt{2} \approx 1.41$.

\paragraph{\textbf{Oracle-Based Expanded Ball Cut.}}
From \Cref{prop:shiftprop} we have equal eigenvalues, and so the shifting has no effect, i.e. $X_C=\bar X$ and $q=1$.  $D^{(1)}$ has finite intersection with $\mbox{cone}(\mathcal{B}(X_C,q))$ since $r_1\approx 0.35 \allowbreak < d_1\approx 0.79$. 
The quadratic equation for determining steplength is $-\lambda_1^2 + 2\lambda_1 + 4 = 0$, so $\lambda_1 = 2\phi$.  

For $D^{(2)}$ we have $r_1\approx 0.35 \allowbreak < d_1\approx 0.79$, so there is a finite step length. The quadratic equation is $-\lambda_2^2 + 2\lambda_2 + 4 = 0$, so $\lambda_2 = 2\phi$. 

For $D^{(3)}$ we have $r_1=0$ and $d_1=1/\sqrt{2}$, so there is a finite step length. The quadratic equation is $-\lambda_3^2 + 4 = 0$, so $\lambda_3 =2$. 

Thus the strengthening recovers the $2\times 2$ cut.  Note that \Cref{thm:smallmaxfree} says that in $\mathbb{S}^{2\times 2}$ both oracle-based outer-product-free sets are contained in the $2\times 2$ cone since they contain the strictly positive definite $\bar X$.  Hence the $2\times 2$ cone is the best possible outer-product-free extension of the oracle ball.

\subsection{Example: Cardinality Constraint}
\label{subsec:card}
The oracle (intersection) cut of \Cref{sec:ballcut} can be computed quickly provided $d(\bar x,S)$ can be determined quickly for a given set $S$. This is the case when $S$ represents $k$-cardinality constrained vectors:
\[S:=\{x\in\mathbb{R}^n\, |\, \mbox{card}(x) \leq k\},\]
where $\mbox{card}(\cdot)$ is the number of nonzero entries. For a given vector $\bar x$, the nearest point in $S$ is a vector $\hat x$ equal to $\bar x$ at the $k$ largest magnitude entries, and zero elsewhere.  

A simple application is the statistical problem of cardinality-constrained least-absolute deviation regression (see \cite{konno2009choosing}):
\[\min\{ |Ax-b| \, : \, \mbox{card}(x)\leq k\}.\]
The cardinality constraint is used to prevent statistical overfitting.  In our example let
\[A=\left[\begin{array}{ccc}
1 & 2 & 3\\
2 & -1 & 1\\
3 & 0 & -1
\end{array}\right], b = \left[\begin{array}{c}
9\\
8\\
3
\end{array}\right],
\]
and let $k=2$.  For the $S$-free approach we use the following formulation,
\begin{subequations}
\begin{align}
\min \qquad x_4 + x_5 + x_6 &  \nonumber\\
s.t. \qquad x_1 + 2x_2 + 3x_3 - 9 &\ \leq \  x_4,\label{eq:excard1}\\
-x_1 - 2x_2 - 3x_3 + 9 &\ \leq \  x_4,\label{eq:excard2}\\
2x_1 - x_2 + x_3 - 8 &\ \leq \  x_5,\label{eq:excard3}\\
-2x_1 + x_2 - x_3 + 8 &\ \leq \  x_5,\label{eq:excard4}\\
3x_1 - x_3 - 3 &\ \leq \  x_6,\label{eq:excard5}\\
-3x_1 + x_3 + 3 &\ \leq \  x_6,\label{eq:excard6}\\
 x_4,x_5,x_6 &\ \geq \  0, \label{eq:excardnn}\\
\mbox{card}([x_1,x_2,x_3])&\ \leq \  2. \label{eq:excardcard}
\end{align}
\end{subequations}

The problem is formulated in extended space with augmented variables $x_4,x_5,x_6$ in order to represent the objective function with linear constraints.  Constraints~(\ref{eq:excard1})-(\ref{eq:excard6}) relate the augmented variables to the original objective function. Constraints~(\ref{eq:excardnn}) are redundant inequalities used to form a simplicial cone after solving the linear programming relaxation.  Dropping the nonconvex constraint~(\ref{eq:excardcard}) yields a linear programming relaxation.

The LP relaxation has an optimal solution $x^* = [2, -1, 3, 0, 0, 0]^T$ with objective value $0$. The closest vector to $x^*$ obeying the cardinality constraint is $[2, 0, 3, 0, 0, 0]^T$ with Euclidean distance 1, giving us a step size of 1 along all directions for the intersection cut.  The simplicial cone with apex $x^*$ may be written as $\bar A x \leq \bar b$, where

\[\bar A=\left[\begin{array}{cccccc}
1 & 2 & 3 & -1 & 0 & 0\\
2 & -1 & 1& 0 & -1 & 0\\
3 & 0 & -1& 0 & 0 & -1\\
0 & 0& 0& -1& 0& 0\\
0 & 0& 0& 0& -1& 0\\
0 & 0& 0& 0& 0& -1
\end{array}\right], \bar b = \left[\begin{array}{c}
9\\
8\\
3\\
0\\
0\\
0
\end{array}\right].
\]
Applying \Cref{eq:closedformcut}, we generate the cut
\[6x_1+x_2+3x_3-2x_4-2x_5-2x_6\leq 19.\]
After adding the cut, the linear programming relaxation has optimal solution 
\(x^* \allowbreak \approx \allowbreak [1.98,-1.08,2.95,0.33,0,0]^T\)
with improved objective value $\frac{1}{3}$.

\subsection{Numerical Experiments}
\label{sec:cexp}

We present experiments using a pure cutting-plane algorithm using 
the cuts described in \Cref{sec:polycuts}. The experiments are designed to investigate the stand-alone performance of our cuts, particularly speed and relaxation quality. 
The cutting plane algorithm solves an LP relaxation and obtains an (extreme point) optimal solution $\bar{X}$, adds cuts 
separating $\bar{X}$, and repeats until either:
\begin{itemize}
\item A time limit of 600 seconds is reached, or
\item The objective value does not improve for 10 iterations, or
\item The violation of all cuts is not more than $10^{-6}$. Here, if $\pi^T x \leq \pi_0 $ is the cut and $x^*$ is the candidate solution, we define the violation as $(\pi^T x^* - \pi_0)/\| \pi \|_1$.
\end{itemize}

For numerical stability, we add a maximum of 20 cuts per iteration (selected using violations) 
and purge non-active cuts every 15 iterations. Computations are run on a 32-core server with an Intel Xeon Gold 6142 2.60GHz CPU and 512 GB of RAM. Although the machine is powerful, we run the algorithm single-threaded and the experiments do not require a significant amount of memory; we confirmed 
that similar performance can be obtained with a laptop. The code is written in C++ using the Eigen library for linear algebra \cite{eigenweb}. The LP solver is Gurobi 8.1.1 
and, for comparisons, we solve SDP relaxations using the C++ Fusion API of Mosek 8 \cite{mosek}. Our code is available at \url{https://github.com/g-munoz/poly_cuts_cpp}.

Our cuts can accommodate polynomials of arbitrary degree, however for implementation purposes reading quadratically-constrained quadratic programs (QCQP) problems is more convenient. Thus, our implementation is built for QCQPs only. However, we note that for our purposes this is without loss of generality.  Any polynomial optimization problem of degree $d$ that is lifted to \textbf{LPO} with monomials of degree $r\geq \lceil d/2 \rceil$ can be transformed to a QCQP whose lifted representation with $r=1$ is the same. One can obtain said QCQP by projecting \textbf{LPO} onto $m_r$: adding all linear consistency constraints on $X$ in \textbf{LPO}; replacing $X$ with $m_rm_r^T$; and treating $m_r$ as a vector of decision variables.


\subsubsection{Experiments in QCQPs}
Our main experiments are performed over nonconvex QCQPs. In this case, test instances are taken from two sources.  First, we consider all 27 problem instances from Floudas et al. \cite{floudas2013handbook} (available via GLOBALLib \cite{globalliburl}) that have quadratic objective and constraints.  Second, we consider all 99 instances of BoxQP developed by several authors \cite{VanNem05b,Burer10}.  These problems have box constraints $x \in [0,1]^n$ and a nonconvex quadratic objective function. We refer the reader to \cite{burer2009globally} for a semidefinite programming approach to this class of instances and \cite{chen2012globally} for a completely positive approach to QPs. Recent papers \cite{bonamiglobally,xia2019globally} have considered solving BoxQP with integer linear programming.

In order to evaluate the performance of our general-purpose cuts, we compare our results with the V2 setting used by Saxena, Bonami and Lee \cite{saxena2010convex} and with SDP relaxations. We chose V2 in \cite{saxena2010convex} as a comparison as we find it the most similar to our approach. V2 uses an lifted linear relaxation for QCQPs and applies two types of cuts:
an outer-approximation of the PSD cone and disjunctive cuts for which the separation involves a MIP.
We emphasize that these families of cuts are complementary and not competitive, and the comparison is only meant to provide a reference on the effectiveness of our cuts.  

In the GLOBALLib instances, we choose the initial LP relaxation to be the standard RLT relaxation of QCQP: setting $r=1$ in $\mathbf{LPO}$ and including McCormick estimators for bilinear terms (see \cite{mccormick1976computability,anstreicher2009semidefinite}). To obtain variable bounds for some of the GLOBALLib instances we apply a simple bound tightening procedure: minimize/maximize a given variable subject to the RLT relaxation. Problem sizes vary from $6\times6$ to $63\times63$ for these instances.

In the BoxQP instances, we adopt for comparison purposes the initial relaxation used by Saxena, Bonami and Lee \cite{saxena2010convex}, namely the weak RLT relaxation (wRLT)\footnote{This BoxQP relaxation only adds the ``diagonal'' McCormick estimates $X_{ii} \leq x_i$.}. Problem sizes vary from $21\times21$ to $126\times 126$ symmetric matrices of decision variables for BoxQP instances.

Lastly, we use \emph{Gap Closed} as a measure of quality of the bounds generated by each approach. This is defined as follows: let $OPT$ denote the optimal value of an instance, $IR$ the optimal value of the initial linear relaxation, and $GLB$ the objective value obtained after applying the cutting plane procedure. Then \( \text{Gap Closed} = \frac{GLB-IR}{OPT-IR}.\)\\

\noindent\textbf{Results}.
In \Cref{table:resultsGlib}, we show a performance comparison in the selected GLOBALLib instances between our cutting plane algorithm and the RLT-strengthened SDP relaxation (RLT+SDP). Due to the large performance variability and the small number of instances in this case, we omit averages of performance measures.  Furthermore, we do not show results for 2 instances for which the RLT relaxation is tight (no cuts are needed). The results in \Cref{table:resultsGlib} are very encouraging: in all but 4 instances our linear relaxations close more gap than RLT+SDP. Moreover, our simple cutting plane approach (almost) always runs in a few seconds and in most instances closes considerably more gap than the SDP relaxation.

For comparison purposes, we turned off our simple bound tightening routine in order to obtain the same initial relaxation value as V2 (and thus the gaps are different than the ones in \Cref{table:resultsGlib}). Certain GLOBALLib instances still not matching initial bound values are excluded from comparison. On comparable GLOBALLib instances our algorithm terminates with a considerable gap closed on many cases, but it does produce smaller gap closed than V2 on some instances. The advantage of our cuts is that runtimes are substantially shorter. This is expected, as V2 solves a MIP in the cut generation, while our cuts only require eigendecomposition and roots of single-variable quadratics.  Moreover, in most cases where intersection cuts do not perform well, V2 also shows modest performance. It is important to mention that the running times in \Cref{table:comparisonGlib} for V2 correspond to the reports in \cite{saxena2010convex}, published in 2010. While new hardware may improve these times, we believe the conclusions we draw from \Cref{table:comparisonGlib} would not change substantially. 


We note that while, in theory, the outer-approximation cuts alone should close the same amount of gap as RLT+SDP, this does not always hold in practice. Pure cutting plane algorithms (especially with dense cuts) can suffer from numerical instability and careful implementation is key. In addition, GLOBALLib instances can be numerically challenging. In our implementation, we included conservative criteria regarding cut efficacy, stalling detection, among others, in order to ensure valid cutting planes. These issues are handled in a more sophisticated fashion in fully-fledged solvers.

In \Cref{table:comparisonBoxQP} we show a similar comparison for the 42 BoxQP instances reported in \cite{saxena2010convex}. In this case, given that these are randomly generated instances of the same type, we measure average gap closed. We do not present average times, however, since there is a time limit present which is reached on many instances. Since the initial relaxation considered for V2 is wRLT, we compare with the wRLT-strengthened SDP relaxation (wRLT+SDP). On these instances, our cuts \emph{always} perform better than both V2 and wRLT+SDP. The latter reaches optimality in seconds, but the relaxation is not strong, as there are missing McCormick inequalities. Our intersection cuts, with a time limit of 600 seconds, are able to close 91.39\% gap on average in these instances, while V2 closes 65.28\% and wRLT+SDP 51.87\%. Despite using a weaker initial relaxation, wRLT, our cuts close a large amount of gap in a short amount of time. 

Additional experiment data, comparing our cuts on larger BoxQP instances with the wRLT+SDP relaxation, can be found in \Cref{appendix:boxqp}. The dimensions of these instances range from $50\times 50$ to $125\times 125$. In these larger instances, the LP relaxations quickly become the bottleneck and 600 seconds is not enough to perform a considerable number of cut rounds. Thus, we also report experiments with 1 hour time limit. Using the latter limit, we obtain a better closed gap than wRLT+SDP on average (48.28\% vs 40.65\%), but the difference is not as substantial as before. Nonetheless, this is not due to the separation procedure itself ---the cut generation remains efficient and effective. Issues arise due to the straightforward initial relaxation we use, which creates too many lifted variables and results in large lifted LP relaxations. As a matter of ongoing work, we are considering the use of partial lifting, exploiting problem sparsity to derive a smaller initial LP relaxation.

\begin{table}
\centering
\caption{Comparison of intersection cuts and RLT+SDP on nonconvex quadratic GLOBALLib instances.} \label{table:resultsGlib}\vspace{-0.4cm}
\scalebox{0.85}{
\begin{tabular}[t]{|l|rr|rr|}
\hline
Instance & \multicolumn{2}{c|}{RLT+SDP} &  \multicolumn{2}{c|}{Intersection Cuts} \\ \cline{2-5}
Name & Gap Closed & Time & Gap Closed & Time \\ \hline
Ex2\_1\_1        & 0.00\%     & 0.01 & 52.90\%        & 0.02     \\
Ex2\_1\_5        & 0.00\%    & 0.02 & 99.57\%        & 0.01     \\
Ex2\_1\_6        & 0.00\%     & 0.02 & 94.34\%        & 0.10     \\
Ex2\_1\_7        & 0.00\%     & 0.28 & 38.44\%        & 0.86     \\
Ex2\_1\_8        & 0.00\%     & 0.62 & 55.89\%        & 2.81     \\
Ex2\_1\_9        & 0.00\%     & 0.02 & 30.20\%        & 0.96     \\
Ex3\_1\_1        & 0.00\%     & 0.02 & 1.29\%         & 2.79     \\
Ex3\_1\_2        & 22.41\%    & 0.01 & 100.00\%       & 0.01     \\
Ex3\_1\_4        & 0.00\%     & 0.01 & 34.64\%        & 0.02     \\
Ex5\_2\_2\_case1 & 0.00\%     & 0.02 & 9.82\%        & 0.43     \\
Ex5\_2\_2\_case2 & 0.00\%     & 0.02 & 0.22\%        & 0.89    \\
Ex5\_2\_2\_case3 & 0.00\%  & 0.02   & 1.07\% & 0.49  \\
Ex5\_2\_4        & 0.00\%  & 0.01   & 28.99\% & 0.23  \\ 
%
Ex5\_2\_5        & 0.00\%  & 3.39   & 0.00\%  & 7.50  \\
Ex5\_3\_2        & 0.10\%  & 0.54   & 0.00\%  & 1.08  \\
Ex5\_3\_3        & 3.75\%  & 91.47 & 0.59\%  & 602.33  \\
Ex5\_4\_2        & 0.00\%  & 0.03   & 0.52\%  & 4.66  \\
Ex8\_4\_1        & 98.43\% & 0.41   & 58.82\% & 38.59 \\
Ex9\_1\_4        & 0.00\%  & 0.04   & 66.33\% & 1.24  \\
Ex9\_2\_1        & 6.25\%  & 0.03   & 34.28\% & 7.69  \\
Ex9\_2\_2        & 16.67\% & 0.03   & 85.96\% & 1.45  \\
Ex9\_2\_3        & 0.00\%  & 0.13   & 0.00\%  & 0.47  \\
Ex9\_2\_4        & 99.83\% & 0.03   & 0.00\% & 0.12  \\
Ex9\_2\_6        & 99.76\% & 0.15   & 99.76\% & 600.01  \\
Ex9\_2\_7        & 6.25\%  & 0.01   & 34.28\% & 7.66  \\\hline
\end{tabular}}
\end{table}
\begin{table}
\centering
\caption{Comparison of intersection cuts and V2 of
\cite{saxena2010convex} on nonconvex quadratic GLOBALLib instances. Entries labelled NR were not reported in \cite{saxena2010convex}.} \label{table:comparisonGlib}
\scalebox{0.85}{
\begin{tabular}[t]{|l|rr|rr|}
\hline
Instance & \multicolumn{2}{c|}{V2} &  \multicolumn{2}{c|}{Intersection Cuts} \\ \cline{2-5}
Name         & Gap Closed & Time & Gap Closed & Time \\
\hline
Ex2\_1\_1        & 72.62\%       & 704.40  & 52.90\%        & 0.02     \\
Ex2\_1\_5        & 99.98\%       & 0.17    & 99.68\%        & 0.00     \\
Ex2\_1\_6        & 99.95\%       & 3397.65 & 86.90\%        & 0.12     \\
Ex2\_1\_8        & 84.70\%       & 3632.28 & 16.46\%        & 6.86     \\
Ex2\_1\_9        & 98.79\%       & 1587.94 & 30.28\%        & 0.97     \\
Ex3\_1\_1        & 15.94\%       & 3600.27 & 1.21\%         & 93.22    \\
Ex3\_1\_2        & 99.99\%       & 0.08    & 100.00\%       & 0.00     \\
Ex3\_1\_4        & 86.31\%       & 21.26   & 34.64\%        & 0.02     \\
Ex5\_2\_2\_case1 & 0.00\%        & 0.02    & 25.71\%        & 5.97     \\
Ex5\_2\_2\_case2 & 0.00\%        & 0.05    & 0.00\%         & 0.14     \\
Ex5\_2\_2\_case3 & 0.36\%        & 0.36    & 25.01\%        & 0.74     \\ 
%
Ex5\_2\_4        & 79.31\%       & 68.93   & 29.35\%        & 0.21     \\
Ex5\_2\_5        & 6.27\%        & 3793.17 & 0.00\%         & 6.50     \\
Ex5\_3\_2        & 7.27\%        & 245.82  & 0.00\%         & 1.05     \\
Ex5\_3\_3        & 0.21\%        & 3693.76 & 0.19\%         & 601.19   \\
Ex5\_4\_2        & 27.57\%       & 3614.38 & 1.84\%         & 161.09   \\
Ex9\_1\_4        & 0.00\%        & 0.60    & 0.00\%         & 0.09     \\
Ex9\_2\_1        & 60.04\%       & 2372.64 & 49.80\%        & 3.31     \\
Ex9\_2\_2        & 88.29\%       & 3606.36 & 73.63\%        & 10.69    \\
Ex9\_2\_6        & 87.93\%       & 2619.02 & 99.84\%        & 600.03   \\
Ex9\_2\_8        & NR             & NR       & 100.00\%       & 0.01 \\ \hline
\end{tabular}}
\end{table}

\begin{table}[]
\centering
\caption{Comparison of intersection cuts and V2 of
\cite{saxena2010convex} and SDP on BoxQP instances (all with wRLT).} \label{table:comparisonBoxQP}
\scalebox{0.85}{
\begin{tabular}{|l|cc|cc|cc|}
\hline
Instance & \multicolumn{2}{c|}{V2}  &  \multicolumn{2}{c|}{Intersection Cuts}  & \multicolumn{2}{c|}{wRLT+SDP}     \\  \cline{2-7}
  Name            & Gap Closed & Time              & Gap Closed & Time  & Gap Closed & Time \\ \hline
spar020-100-1 & 95.40\%    & 3638.2            & 99.93\%    & 11.9  & 58.66\%    & 0.4  \\ 
spar020-100-2 & 93.08\%    & 3636.7            & 96.59\%    & 25.3  & 70.36\%    & 0.3  \\ 
spar020-100-3 & 97.47\%    & 3632.6            & 100.00\%   & 1.3   & 70.70\%    & 0.3  \\ 
spar030-060-1 & 60.00\%    & 3823.1            & 82.87\%    & 434.4 & 35.07\%    & 2.3  \\ 
spar030-060-2 & 91.16\%    & 3716.0            & 100.00\%   & 52.2  & 67.05\%    & 2.1  \\ 
spar030-060-3 & 77.41\%    & 3696.5            & 94.02\%    & 348.2 & 55.50\%    & 2.1  \\
spar030-070-1 & 57.39\%    & 3786.0            & 76.34\%    & 310.5 & 32.29\%    & 2.3  \\ 
spar030-070-2 & 86.60\%    & 3708.2            & 99.42\%    & 172.5 & 63.54\%    & 2.5  \\ 
spar030-070-3 & 88.66\%    & 3744.0            & 99.32\%    & 94.4  & 75.51\%    & 1.9  \\ 
spar030-080-1 & 69.67\%    & 3600.8            & 85.52\%    & 294.6 & 43.19\%    & 2.4  \\
spar030-080-2 & 86.25\%    & 3627.1            & 100.00\%   & 34.5  & 55.24\%    & 2.1  \\
spar030-080-3 & 91.42\%    & 3666.4            & 100.00\%   & 32.1  & 71.15\%    & 2.2  \\
spar030-090-1 & 81.15\%    & 3676.8            & 95.06\%    & 171.2 & 55.95\%    & 2.0  \\
spar030-090-2 & 82.66\%    & 3646.8            & 98.98\%    & 206.3 & 58.08\%    & 2.5  \\
spar030-090-3 & 86.37\%    & 3701.8            & 100.00\%   & 69.7  & 61.82\%    & 2.0  \\
spar030-100-1 & 81.10\%    & 3692.5            & 95.51\%    & 600.0 & 60.99\%    & 2.3  \\
spar030-100-2 & 72.87\%    & 3697.3            & 92.35\%    & 525.1 & 51.76\%    & 2.5  \\
spar030-100-3 & 84.10\%    & 3606.5            & 95.16\%    & 230.4 & 64.38\%    & 1.9  \\
spar040-030-1 & 31.05\%    & 3719.2            & 83.79\%    & 684.7 & 25.52\%    & 10.0 \\
spar040-030-2 & 27.74\%    & 3937.9            & 84.14\%    & 608.1 & 26.23\%    & 10.1 \\
spar040-030-3 & 28.00\%    & 3798.7            & 77.87\%    & 609.6 & 9.94\%     & 8.6  \\
spar040-040-1 & 33.31\%    & 3817.8            & 66.88\%    & 601.7 & 23.22\%    & 9.5  \\
spar040-040-2 & 35.19\%    & 3968.1            & 93.85\%    & 602.2 & 38.45\%    & 9.8  \\
spar040-040-3 & 26.71\%    & 3972.9            & 75.09\%    & 608.5 & 22.81\%    & 9.6  \\
spar040-050-1 & 36.72\%    & 3819.7            & 79.24\%    & 602.8 & 30.88\%    & 11.4 \\
spar040-050-2 & 40.87\%    & 3610.6            & 85.10\%    & 610.9 & 35.95\%    & 10.3 \\
spar040-050-3 & 33.95\%    & 3640.0            & 82.09\%    & 604.8 & 30.13\%    & 10.5 \\
spar040-060-1 & 47.75\%    & 3761.0            & 82.65\%    & 601.3 & 42.64\%    & 9.5  \\
spar040-060-2 & 55.79\%    & 3708.0            & 94.64\%    & 610.5 & 54.28\%    & 8.7  \\
spar040-060-3 & 72.63\%    & 3764.1            & 99.34\%    & 601.9 & 65.22\%    & 8.1  \\
spar040-070-1 & 64.03\%    & 3642.7            & 93.07\%    & 612.0 & 60.32\%    & 9.4  \\
spar040-070-2 & 57.91\%    & 3756.4            & 94.96\%    & 602.0 & 53.83\%    & 8.4  \\
spar040-070-3 & 62.94\%    & 3693.7            & 94.46\%    & 600.9 & 58.81\%    & 8.9  \\
spar040-080-1 & 58.37\%    & 3808.3            & 87.61\%    & 602.0 & 49.34\%    & 8.9  \\
spar040-080-2 & 66.96\%    & 4062.4            & 95.18\%    & 600.9 & 57.79\%    & 8.5  \\
spar040-080-3 & 72.31\%    & 4057.1            & 96.63\%    & 602.4 & 67.45\%    & 9.6  \\
spar040-090-1 & 66.64\%    & 3781.0            & 91.18\%    & 603.9 & 60.26\%    & 7.6  \\
spar040-090-2 & 66.46\%    & 3931.3            & 92.12\%    & 602.2 & 60.78\%    & 8.7  \\
spar040-090-3 & 73.49\%    & 4003.7            & 96.87\%    & 602.2 & 66.45\%    & 8.1  \\
spar040-100-1 & 76.24\%    & 3853.6            & 97.23\%    & 602.4 & 70.05\%    & 8.4  \\
spar040-100-2 & 63.89\%    & 3658.3            & 92.98\%    & 607.4 & 59.42\%    & 8.5  \\
spar040-100-3 & 59.92\%    & 3842.7            & 90.41\%    & 602.9 & 57.40\%    & 8.6  \\ \hline \hline
Average       & 65.28\%    &                   & 91.39\%    &       & 51.87\%    &  \\ \hline  
\end{tabular}}
\end{table}

\subsubsection{Preliminary experiments on polynomial instances}
We also performed some preliminary experiments in polynomial optimization instances. We handled these instances as we mention above: we use additional variables to obtain an equivalent QCQP and use the same setting as in the previous section.

In this case, instances are drawn from two sources. Firstly, as in the previous section, we consider instances from \cite{floudas2013handbook} (available via GLOBALLib \cite{globalliburl}). The second set of instances are the ones labeled \emph{miscellaneous} in \cite{tawarmalani2002convexification}. From these sources we consider instances with polynomials of degree 3 or more, that have more than 1 and less than 100 variables, and whose initial lifted linear relaxation is not unbounded. This leaves 11 instances with 2--24 variables and polynomial degrees ranging between 3--8. The limitation on the number of variables was set to avoid memory issues in both our method and in Mosek due to the rapid increase in dimension in moment-based approaches. As we mention above, a sparse version of our approach is the subject of current work. 


In \Cref{table:polyopt} we present our computational results. In this table we compare the strength of the dual bound produced by our method with the SDP+RLT relaxation solved with Mosek. We can observe that our cutting planes always dominate the SDP-based method except in one instance. As we mentioned before, this is expected but does not necessarily hold in practice due to our numerical safety measures. In this case, however, the dominance of our cuts is more modest than in the quadratic case. Nonetheless, we believe these results are promising, as we are still able to do strictly better in some instances with our direct implementation. 

It it worth mentioning that these instances are much more challenging from the numerical perspective, and we believe this is the main reason why our cuts are not as dominant as before. For example, in instance \emph{ex8\_4\_2}, many cutting planes are discarded by our stability criteria, and the algorithm quickly stops when detecting odd oscillations in the dual bounds produced by each round of cuts. We believe the increased instability arises from the equivalent QCQP formulation (or, equivalently, the linear consistency constraint among monomials). This adds many linear \emph{equality} constraints that reduce the dimension of the lifted polyhedral relaxations, which can create issues in the intersection cut framework if not treated with care. Furthermore, high degree polynomials may exacerbate numerical issues; for instance, replacing $y=x^8$ with a tower of variables, $w_1=x^2, w_2=w_1^2, y=w_2^2$, could result in an accumulation of errors.  We expect that developing a version of these cuts that can better handle dimension increase and numerical instability will provide a powerful alternative to complement existing methods. 

\begin{table}[]
\centering
\caption{Comparison of intersection cuts and RLT+SDP on polynomial optimization instances.} \label{table:polyopt}
\begin{tabular}{|l|rr|rr|}
\hline
Instance  &\multicolumn{2}{c|}{RLT+SDP} &  \multicolumn{2}{c|}{Intersection Cuts} \\ \cline{2-5}
Name  & Gap Closed & Time & Gap Closed & Time \\ \hline
ex4\_1\_8 & 100.00\% & 0.01 & 100.00\% & 0.03 \\
ex4\_1\_9 & 0.00\%   & 0.02  & 16.59\%  & 0.02 \\
ex7\_3\_1 & 0.00\%   & 0.05  & 0.00\%   & 0.55 \\
ex7\_3\_2 & 0.00\%   & 0.01  & 0.00\%   & 0.01 \\
ex8\_1\_4 & 100.00\% & 0.01  & 100.00\% & 0.04 \\
ex8\_1\_5 & 98.72\%  & 0.02  & 98.90\%  & 0.29 \\
ex8\_1\_7 & 55.79\%  & 0.03  & 66.54\%  & 180.50 \\
ex8\_4\_2 & 90.73\%  & 8.23  & 0.00\%  & 0.43 \\
st\_e03   & 31.64\%  & 0.31  & 81.52\%  & 4.52 \\
st\_e10   & 100.00\% & 0.01  & 100.00\% & 0.07 \\
st\_e19   & 97.47\%  & 0.01  & 99.13\%  & 6.45 \\\hline
\end{tabular}
\end{table}

\section{Conclusions}
\label{sec:conc}
We have introduced cuts for the generic set $S\cap P$, where for the closed set $S$ there is an oracle that provides the distance from a point to the nearest point in $S$. We have shown that the oracle can be used to construct a convergent cutting plane algorithm that can produce arbitrarily close approximations to $\mbox{conv}(S\cap P)$ in finite time. This algorithm relies on a (potentially) computationally expensive cut generation procedure, and so we have also considered a simple oracle-based intersection cut that can be easily computed. We provide applications of this intersection cut on polynomial optimization problems as well as a cardinality-constrained problem.  Furthermore, we provide a generic strengthening procedure of the intersection cut that uses the recession cone of an $S$-free set.

We have also introduced intersection cuts in the context of polynomial optimization. Accordingly, we have developed an $S$-free approach for polynomial optimization, where $S$ is the set of real, symmetric outer products. Our results on full-dimensional maximal OPF sets include a full characterization of such sets when $n_r=2$ as well as extensive families of maximal OPF sets. We derived intersection cuts from these families of maximal outer-product-free sets, including a strengthening procedure that determines negative step lengths in the case of intersections at infinity. 

Computational experiments have demonstrated the potential of our cuts as a fast way to reduce optimality gaps on a variety of polynomial optimization problems. We note that, although such experiments contrast our cuts with SDP and V2, the methods are in fact complementary. For instance, SDP can be used to warm-start outer-approximation cuts, and our cuts can, in turn, be added back to strengthen SDP (this also holds true for V2). A full implementation is being considered for future empirical work, incorporating the cuts into a branch-and-cut solver and developing a more sophisticated implementation, e.g. stronger initial relaxations with problem-specific valid inequalities, exploiting sparsity, advanced cut management, improved scalability, among others.

\subsubsection*{Acknowledgements}
The authors thank Eli Towle for pointing out an error in the presentation of the intersection cut strengthening procedure, Felipe Serrano for useful comments and suggestions that led to \Cref{lemma:22always}, and to the anonymous reviewers whose thorough feedback greatly improved the article. The authors would also like to thank the Institute for Data Valorization (IVADO) for their support through the IVADO Postdoctoral Fellowship program.

\bibliography{references} 
\bibliographystyle{spmpsci}
\newpage
\appendix
\normalsize
\section*{Appendix}
\section{Radius of the Conic Hull of a Ball}
\label{sec:radius}
Suppose we have a ball of radius $r$ and with centre that is distance $m>r$ from the origin. We wish to determine the radius of the conic hull of the ball at a specific point along its axis.  Consider a 2-dimensional cross-section of the conic hull of the ball containing the axis; this is shown in \Cref{fig:coneball} in rectangular $(x,y)$ coordinates.  A line passing through the origin and tangent to the boundary of the ball in the nonnegative orthant may be written in the form $y=ax$ for some $a>0$; let $(\bar r, \bar m)$ be the point of intersection between line and ball.  At $(\bar r, \bar m)$ we have 
\begin{equation}
\label{eq:apprad1}
(a\bar r-m)^2+\bar r^2=r^2 \iff (1+a^2)\bar r^2-2am\bar r + m^2-r^2 = 0.
\end{equation}

Now \Cref{eq:apprad1} should only have one unique solution with respect to $\bar r$ since the line is tangent to the ball; thus the discriminant must be zero,

\begin{equation}
\label{eq:apprad3}
4a^2m^2-4(1+a^2)(m^2-r^2)=0 \implies a=\frac{\sqrt{m^2-r^2}}{r}.
\end{equation}
Solving \Cref{eq:apprad1} for $\bar r$ with \Cref{eq:apprad3},
\begin{alignat*}{2}
&\bar r \ &=&\ \frac{2am}{2(1+a^2)}, \\
& &=&\ \frac{r}{m}\sqrt{m^2-r^2},\\
&\bar m \ &=& \ a\bar r, \\
& &=& \ \frac{m^2-r^2}{m}.
\end{alignat*}
Hence at distance $d$ from the origin along the axis of the cone, the radius of the cone is $\frac{\bar r}{\bar m}d$, or
\begin{equation}
\label{eq:appradfinal}
\frac{r}{\sqrt{m^2-r^2}}d.
\end{equation}

\begin{figure}
  \centering
   \includegraphics[width=0.4\textwidth]{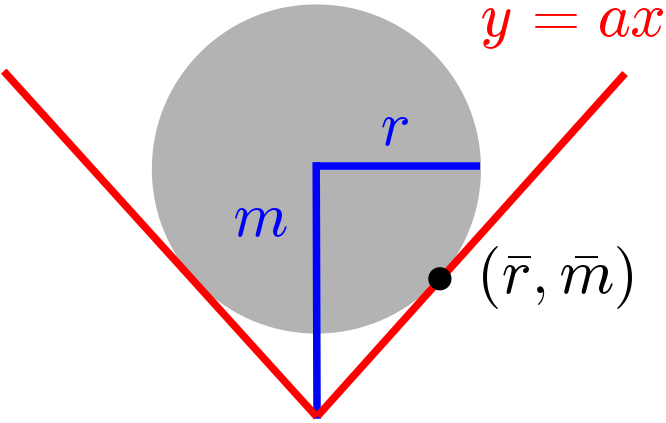}
   \caption{In grey, a ball with radius $r$ and distance $m>r$ from the origin. In red, the boundary of its conic hull. In black, an intersection point between the boundary of the ball and its conic hull.}
\label{fig:coneball}
\end{figure}

\section{Additional BoxQP Experiments} \label{appendix:boxqp}
\begin{center}
\begin{table}[h!]
\caption{Comparison of intersection cuts and wRLT+SDP on larger BoxQP instances.} \label{table:comparisonBoxQP_extended}
\scalebox{0.85}{
\begin{tabular}{|l|cc|cc|cc|}
\hline
Instance  &  \multicolumn{2}{c|}{Intersection Cuts (10 min)} & \multicolumn{2}{c|}{Intersection Cuts (1h)} & \multicolumn{2}{c|}{wRLT+SDP}     \\  \cline{2-7}
  Name   & Gap Closed & Time &  Gap Closed & Time & Gap Closed & Time \\ \hline
spar050-030-1 & 53.86\% & 608.8  & 75.60\% & 3604.97 & 19.12\% & 24.3   \\
spar050-030-2 & 39.73\% & 601.1  & 65.51\% & 3606.45 & 11.34\% & 26.9   \\
spar050-030-3 & 44.81\% & 919.5  & 71.64\% & 3604.77 & 10.72\% & 28.5   \\
spar050-040-1 & 69.61\% & 619.4  & 83.99\% & 3605.79 & 35.79\% & 28.7   \\
spar050-040-2 & 61.82\% & 600.8  & 79.89\% & 3656.72 & 29.33\% & 30.2   \\
spar050-040-3 & 78.07\% & 601.9  & 93.50\% & 3602.42 & 37.17\% & 25.9   \\
spar050-050-1 & 56.68\% & 636.8  & 65.82\% & 3602.11 & 27.05\% & 30.5   \\
spar050-050-2 & 65.38\% & 608.1  & 76.55\% & 3645.30 & 35.61\% & 29.7   \\
spar050-050-3 & 71.90\% & 604.2  & 83.73\% & 3605.17 & 44.44\% & 30.4   \\
spar060-020-1 & 9.99\%  & 663.6  & 21.45\% & 3618.78 & 5.86\%  & 73.2   \\
spar060-020-2 & 15.31\% & 620.3  & 36.21\% & 3696.31 & 9.54\%  & 77.5   \\
spar060-020-3 & 9.28\%  & 603.8  & 25.18\% & 3616.49 & 9.33\%  & 72.4   \\
spar070-025-1 & 18.84\% & 631.3  & 31.03\% & 4407.37 & 19.70\% & 177.7  \\
spar070-025-2 & 7.61\%  & 600.3  & 17.84\% & 3695.73 & 10.19\% & 186.2  \\
spar070-025-3 & 11.71\% & 635.5  & 32.77\% & 3682.20 & 14.89\% & 191.8  \\
spar070-050-1 & 50.38\% & 627.6  & 63.66\% & 3602.71 & 44.42\% & 225.7  \\
spar070-050-2 & 52.75\% & 617.6  & 66.25\% & 3626.46 & 42.50\% & 208.6  \\
spar070-050-3 & 57.12\% & 625.1  & 78.45\% & 3616.35 & 54.30\% & 177.5  \\
spar070-075-1 & 64.96\% & 630.3  & 77.09\% & 3639.95 & 59.15\% & 176.1  \\
spar070-075-2 & 61.18\% & 611.8  & 75.05\% & 3601.85 & 57.71\% & 179.6  \\
spar070-075-3 & 65.75\% & 614.4  & 77.71\% & 3601.76 & 58.13\% & 209.8  \\
spar080-025-1 & 7.96\%  & 629.1  & 22.92\% & 3857.10 & 14.06\% & 446.9  \\
spar080-025-2 & 5.38\%  & 608.9  & 16.23\% & 3725.68 & 14.00\% & 433.7  \\
spar080-025-3 & 7.21\%  & 631.7  & 23.95\% & 3956.51 & 18.85\% & 430.4  \\
spar080-050-1 & 42.52\% & 615.1  & 56.58\% & 3708.40 & 45.42\% & 463.3  \\
spar080-050-2 & 45.86\% & 640.7  & 62.28\% & 3614.02 & 50.72\% & 445.1  \\
spar080-050-3 & 47.64\% & 614.0  & 64.25\% & 3674.85 & 50.50\% & 480.6  \\
spar080-075-1 & 58.54\% & 617.4  & 75.07\% & 3690.96 & 63.84\% & 418.3  \\
spar080-075-2 & 61.64\% & 636.6  & 75.80\% & 3782.86 & 63.79\% & 441.6  \\
spar080-075-3 & 61.31\% & 631.5  & 71.86\% & 3650.25 & 63.13\% & 409.3  \\
spar090-025-1 & 0.77\%  & 644.5  & 14.72\% & 3707.69 & 21.99\% & 891.6  \\
spar090-025-2 & 0.79\%  & 662.3  & 13.76\% & 3745.97 & 21.14\% & 852.0  \\
spar090-025-3 & 0.82\%  & 614.5  & 13.13\% & 3667.08 & 20.52\% & 838.6  \\
spar090-050-1 & 22.17\% & 623.4  & 53.40\% & 3711.48 & 51.53\% & 1077.5 \\
spar090-050-2 & 21.42\% & 635.2  & 57.08\% & 3764.55 & 53.60\% & 975.0  \\
spar090-050-3 & 13.15\% & 649.8  & 58.00\% & 3625.29 & 53.84\% & 980.7  \\
spar090-075-1 & 39.52\% & 646.4  & 66.95\% & 3636.33 & 60.43\% & 908.7  \\
spar090-075-2 & 40.78\% & 618.9  & 66.43\% & 3622.32 & 60.31\% & 870.4  \\
spar090-075-3 & 44.28\% & 655.1  & 67.35\% & 3653.55 & 61.45\% & 968.6  \\
spar100-025-1 & 0.51\%  & 686.9  & 16.20\% & 3708.42 & 26.80\% & 1819.1 \\
spar100-025-2 & 0.40\%  & 605.3  & 12.03\% & 3731.99 & 22.98\% & 1762.3 \\
spar100-025-3 & 1.12\%  & 632.1  & 14.42\% & 3634.96 & 26.58\% & 1492.7 \\
spar100-050-1 & 6.81\%  & 655.3  & 47.61\% & 3716.91 & 50.20\% & 1621.7 \\
spar100-050-2 & 3.84\%  & 673.0  & 48.71\% & 3679.68 & 52.75\% & 1825.4 \\
spar100-050-3 & 8.13\%  & 736.4  & 49.07\% & 3641.59 & 52.57\% & 2066.3 \\
spar100-075-1 & 17.62\% & 732.6  & 67.72\% & 3714.62 & 63.28\% & 1681.3 \\
spar100-075-2 & 7.08\%  & 612.4  & 65.11\% & 3719.54 & 64.38\% & 1630.4 \\
spar100-075-3 & 22.71\% & 722.0  & 62.07\% & 3722.93 & 64.39\% & 1590.9 \\
spar125-025-1 & 0.69\%  & 744.3  & 0.84\%  & 3600.29 & 29.61\% & 3670.7 \\
spar125-025-2 & 0.53\%  & 687.7  & 0.73\%  & 3861.42 & 33.63\% & 3870.9 \\
spar125-025-3 & 0.36\%  & 1001.0 & 0.36\%  & 3718.06 & 33.50\% & 3863.5 \\
spar125-050-1 & 0.11\%  & 865.5  & 22.25\% & 3944.22 & 56.81\% & 3746.6 \\
spar125-050-2 & 0.11\%  & 738.9  & 21.90\% & 3605.57 & 58.53\% & 3762.0 \\
spar125-050-3 & 0.17\%  & 901.2  & 21.67\% & 3789.88 & 57.55\% & 3619.1 \\
spar125-075-1 & 0.00\%  & 631.4  & 45.45\% & 3927.23 & 67.61\% & 3824.3 \\
spar125-075-2 & 4.31\%  & 1209.0 & 36.61\% & 3773.20 & 64.92\% & 3861.4 \\
spar125-075-3 & 1.53\%  & 1224.0 & 40.29\% & 3657.00 & 65.55\% & 3673.8 \\ \hline
\end{tabular}}
\end{table}
\end{center}
\end{document}